\documentclass[10pt]{amsart}
\usepackage{amssymb}
\usepackage{amsmath}
\usepackage{amsthm}
\usepackage{mathrsfs}
\usepackage[all, cmtip]{xy}

\newtheorem{thm}{Theorem}[section]
\newtheorem{lem}[thm]{Lemma}
\newtheorem{prop}[thm]{Proposition}
\newtheorem{cor}[thm]{Corollary}

\newtheorem{defn}[thm]{Definition}

\newcommand{\RR}{\mathbb{R}}
\newcommand{\ZZ}{\mathbb{Z}}
\newcommand{\CC}{\mathbb{C}}

\newcommand{\bul}{\bullet}
\newcommand{\G}{\Gamma}

\let\ms\mathscr
\let\bb\mathbb
\let\cl\mathcal
\let\rm\mathrm
\let\fr\mathfrak

\title[Cohomology of Standard Modules on Partial Flag Varieties]{Cohomology of Standard Modules on \\ Partial Flag Varieties}
\author{S. N. Kitchen}

\begin{document}

{\abstract Cohomological induction gives an algebraic method for constructing
	representations of a real reductive Lie group $G$ from irreducible representations of 
	reductive subgroups.  Beilinson-Bernstein localization alternatively gives a geometric
	method for constructing Harish-Chandra modules for $G$ from certain representations 
	of a Cartan subgroup.   
	The duality theorem of Hecht, Mili\v ci\'c, Schmid and Wolf establishes
	 a relationship between modules cohomologically induced from minimal parabolics  and 
	 the cohomology of the $\ms{D}$-modules on the complex flag variety 
	 for $G$ determined by the Beilinson-Bernstein construction.  The main results
	 of this paper give a generalization of the duality theorem to partial flag varieties, 
	 which recovers cohomologically induced modules arising from nonminimal parabolics.
}

 \maketitle

%%%%%%
\section{Introduction}
%%%%%%

	The objective of this paper is to extend the duality theorem of \cite{hmsw} to 
	partial flag varieties.  For $G_{\bb{R}}$ be a real reductive Lie group 
	and $(\fr{g}, K)$ its complex Harish-Chandra pair, the main difference 
	between the geometry of $K$-orbits on the full flag variety of $\fr{g}$ and 
	$K$-orbits on partial flag varieties is that the orbits are not necessarily 
	affinely embedded in the case of partial flag varieties, whereas they are
	for the full flag variety.  The affineness of the embedding of 
	$K$-orbits in the full flag variety of $\fr{g}$ was used in an essential way in 
	\cite{hmsw}.  Motivated by the derived equivariant constructions of \cite{p} and
	\cite{mp}, we define analogous geometric constructions which allow us to 
	prove our main result using derived category techniques to take into account
	the failure of affineness of $K$-orbit embeddings.

%%%%%
\subsection{Main Theorem}
%%%%%
	
	Before stating our main result, we first recall the duality theorem of \cite{hmsw}.
	As above, let $G_\RR$ denote a real reductive Lie group, to which we 
	associate its complex Harish-Chandra pair  $(\fr{g},K)$ and abstract Cartan triple 
	$(\fr{h},\Sigma,\Sigma^+)$.  On the full flag variety $X$ of $\fr{g}$, let $Q$ be a 
	$K$-orbit and $\tau$ an irreducible connection on $Q$.  There is a twisted sheaf 
	of differential operators $\ms{D}_\lambda$ on $X$ for every $\lambda\in\fr{h}^*$.
	The $\ms{D}_\lambda$-modules on $X$ have cohomology groups which are 
	Harish-Chandra modules with infinitesimal character $[\lambda]\in\fr{h}^*/\cl{W}$.
	When $\tau$ and $\lambda$ are compatible, we define the \emph{standard 
	module} on $X$ corresponding to the pair $(\tau, \lambda)$ to be the 
	$\ms{D}_\lambda$-module direct image $i_+\tau$ of $\tau$ along the inclusion 
	$i: Q\to X$.  Recall for $V$ a $(\fr{b}, L)$-module with $\fr{b}\subset \fr{g}$ a Borel subalgebra and $L$ a subgroup of $K$,  we induce $V$ to a $(\fr{g},L)$-module by taking the tensor product $\rm{ind}_{\fr{b},L}^{\fr{g},L}(V) = \cl{U}(\fr{g})\otimes_{\cl{U}(\fr{b})} V$.  Let $T_x$ denote the geometric fiber functor.  
	
	 We state the main theorem of 
	\cite{hmsw} not in its original form as a duality statement, but instead 
	without contragredients so that it takes a form similar to the natural formulation
	of our main result.  
	
	\begin{thm}[\cite{hmsw}, Theorem 4.3] \label{dualitythm}   Let $x\in Q$ be any 
	point, let $B_x$ be its stabilizer in $G$, and let $\fr{b}_x$ be the Lie algebra of $B_x$.  
	 Put $\fr{n}_x=[\fr{b}_x,\fr{b}_x]$ and let $\bar{\fr{n}}_x$ be its opposite in $\fr{g}$.
	 Then for all $p\in \ZZ$, we have
	\begin{equation*}
		\rm{H}^p(X,i_+\tau) \simeq \rm{\bf R}^{d_Q+p}\G_{K,B_x\cap K}(
		\rm{ind}_{\fr{b}_x,B_x\cap K}^{\fr{g}, B_x\cap K}(T_x\tau\otimes \wedge^{\rm{top}}\bar{\fr{n}}_x))
	\end{equation*}
	as $\cl{U}(\fr{g})$-modules with infinitesimal character $[\lambda]$.
	\end{thm}
\noindent	This theorem shows that the sheaf cohomology
	of standard $K$-equivariant $\ms{D}_\lambda$-modules on the full flag variety 
	$X$ for are isomorphic to cohomologically induced modules; that is,
	modules which are cohomologically induced from Borels.  
	Our main result is the analogous identification of the 
	cohomology of standard $\ms{D}_\lambda$-modules on a partial flag variety 
	$X_\theta$, where $\theta$ is a subset of simple roots, with Harish-Chandra 
	modules cohomologically induced from parabolics of type $\theta$.   
	
Unfortunately, Theorem \ref{dualitythm}.\ fails to generalize immediately to partial 
flag varieties because the direct image functor $i_+$ is not necessarily exact for 
the inclusion of a non-affinely embedded $K$-orbit $Q$ in $X_\theta$.  That is, the 
direct image $i_+\tau$ may be a complex of $\ms{D}_\lambda$ modules
rather than a single sheaf.  Theorem \ref{mainthm}.\ below  is an extension of 
Theorem \ref{dualitythm}.\ which incorporates the possible failure of exactness of 
$i_+$.  Theorem \ref{dualitythm}.\ can be recovered as an immediate corollary.  
Let $X_\theta$ be a partial flag variety for $\fr{g}$ and $p: X\to X_\theta$ 
the natural projection from the full flag variety.    Let $\rho$ be the half-sum 
of roots in $\Sigma^+$, let $\rho_\theta$ be the half-sum of roots in $\Sigma^+$
generated by $\theta$, and define $\rho_n = \rho - \rho_\theta$.  
Our main theorem is then:

\begin{thm}[Main Theorem]\label{mainthm} Let $\ms{D}_\lambda$ be a homogeneous
 tdo on  $X_\theta$ and let $\tau$ be a connection on a $K$-orbit $Q$ compatible 
 with $\lambda+\rho_n$.    For $x\in Q$, let $\fr{p}_x$ be the corresponding parabolic
 in $\fr{g}$, let $\fr{n}_x = [\fr{p}_x,\fr{p}_x]$, and let $S_x$ be the stabilizer of $x$ in $K$.   
 Then there is an isomorphism 
\begin{equation}\label{maindciso}
   \rm{\bf R}\G(X,p^\circ i_+ \tau) \simeq \G^\rm{equi}_{K,S_x}
   (\rm{ind}_{\fr{p}_x,S_x}^{\fr{g},S_x}(T_x\tau\otimes \wedge^{\rm{top}}\bar{\fr{n}}_x))
   [d_Q]
\end{equation}
in $\rm{D^b}(\cl{U}_{[\lambda-\rho_\theta]}, K)$, where $d_Q$ is the dimension of $Q$.  
Upon taking cohomology, there is a convergent spectral sequence 
\begin{equation}\label{specseq}
     \rm{\bf R}^p\G(X,p^\circ \rm{\bf R}^qi_+\tau) \Longrightarrow 
     \rm{\bf R}^{d_Q+p+q}\G_{K,S_x}(\rm{ind}_{\fr{p}_x, S_x}^{\fr{g},S_x}
     (T_x\tau\otimes \wedge^{\rm{top}}\bar{\fr{n}}_x)).     
\end{equation}  
\end{thm}
\noindent In this theorem, the category $\rm{D^b}(\cl{U}_{\chi},K)$ is the equivariant bounded derived category of Harish-Chandra modules with infinitesimal character $\chi$ and $\G^{\rm{equi}}_{K,S}$ is the equivariant Zuckerman functor introduced in \S 3.  

The spectral sequence (\ref{specseq}) collapses in special cases, such as when $X_\theta$ is
the full flag variety, but in general Theorem \ref{mainthm}.\ is 
the closest we get to a direct generalization of Theorem \ref{dualitythm}.  
 However, for applications to composition series computations in the Grothendieck 
 group, the convergence of (\ref{specseq}) is sufficient.  
 
 The idea behind the proof of Theorem \ref{mainthm}.\ is that the standard sheaf
 $i_+\tau$ is determined entirely by the geometric fiber $T_x\tau$ at a point $x\in Q$.
 We make this precise by constructing an essential inverse to the functor $T_x$.  
 In this construction we introduce the geometric Zuckerman functor 
 $\G^\rm{geo}_{K,S}$.   The isomorphism (\ref{maindciso}) follows from 
 the commutivity properties of $\G^\rm{geo}_{K,S}$, together with 
 Theorem \ref{embthmstatement}.\ below, which allows us to identify the $\cl{U}(\fr{g})$-module
 structure on the sheaf cohomology in Theorem \ref{mainthm}.  
 
 \begin{thm}[Embedding Theorem]\label{embthmstatement}
	 The inverse image functor
$		p^\circ: \cl{M}(\ms{D}_\lambda)\to \cl{M}(\ms{D}_{\lambda}^p)$
	is fully faithful for all $\lambda$, and for $\lambda$ anti-dominant, we have
	$\G\circ p^\circ = p^*\circ\G$, where $p^*: \cl{M}(\G(\ms{D}_\lambda)) \to
	\cl{M}(\G(\ms{D}_\lambda^p))$ is the usual pull-back of modules induced by
	the natural map $\G(\ms{D}_\lambda^p)\to \G(\ms{D}_\lambda)$.
\end{thm}

%%%%%%
\subsection{Contents of Paper}
%%%%%%
In $\S 2$-$3$, we review twisted differential operators on homogeneous spaces and the construction of the equivariant  Zuckerman functor $\G^{\rm{equi}}_{K,S}$ of \cite{p}.
The functor $\G^{\rm{equi}}_{K,S}$ is the generalization of the usual derived 
Zuckerman functor to categories of derived equivariant complexes.  In \cite{p}, 
Pand\v zi\'c  proves that by taking cohomology of $\G^{\rm{equi}}_{K,S}$ we recover 
the usual Zuckerman functors. That is, for all $p$ we have
 \begin{equation}\label{equal}
    \rm{H}^p(\G^{\rm{equi}}_{K, T} V^\bul) = \rm{\bf R}^p\G_{K,T}(V^\bul).
  \end{equation}

Section 4 is the technical heart of the paper where we introduce the derived 
equivariant category of Harish-Chandra sheaves, define the \emph{geometric 
Zuckerman functor} $\G^\rm{geo}_{K,S}$ (which is the localization of 
$\G^{\rm{equi}}_{K,S}$), and prove that $\G^\rm{geo}_{K,S}$ has sundry properties 
that will be used in the proof of the Theorem \ref{mainthm}.  
In the final section, we prove Theorems \ref{mainthm}.\ and \ref{embthmstatement}.\ and 
end the paper with a brief reformulation of Theorem \ref{mainthm}.\ as a duality statement.

  %%%%%%
  \subsection{Acknowledgements}
  %%%%%%
  
  I would first like to thank my advisor Dragan Mili\v ci\'c, without whom none of this would have been possible, and Pavle Pand\v zi\'c, whose thesis provided the groundwork for mine.  Special thanks go to Wolfgang Soergel and Andrew Snowden for all of their useful comments.   I would also like to thank the members of the University of Utah math department for their support during my graduate student years and express my gratitude to the NSF, as I was fortunate enough to have been supported by the VIGRE grant for several semesters.     
    
 %%%%%%%%%%%%%%%%%%%%%%%%%%%%%%%%%%%%%%%%%%%%%
 %%%%%%%%%%%%%%%%%%%%%%%%%%%%%%%%%%%%%%%%%%%%%

 %%%%%%%
 \section{Twisted Sheaves of Differential Operators}
 %%%%%%%
 
 In this section, we introduce our notation for the direct and inverse image of
  $\ms{D}$-modules, where $\ms{D}$ is a twisted sheaf of differential operators.  Additionally,
  we give classification results for homogeneous sheaves of twisted differential operators 
  on generalized flag varieties.   We learned much of the material from Mili\v ci\'c's unpublished notes \cite{d-dmods} and \cite{d-book}. 
  
  \subsection{Definitions}\label{tdos}
  
  We will always use $\ms{D}_X$ to denote the sheaf of differential operators on 
  a smooth complex algebraic variety $X$ and more generally $\ms{D}$ for a 
  twisted sheaf of differential operators (tdo); that is, a sheaf of $\ms{O}_X$-algebras
  locally isomorphic to $\ms{D}_X$.  Let $\cl{M}(\ms{D})$ denote the category of 
  left $\ms{D}$-modules and $\rm{D^b}(\ms{D})$ the corresponding bounded 
  derived category.  For right $\ms{D}$-modules
   we write $\cl{M}(\ms{D})_\rm{r}$ and $\rm{D^b}(\ms{D})_\rm{r}$, respectively.  
 
  Fix a smooth map $f: Y\to X$ between smooth varieties and define 
  $\ms{D}^f$ to be the sheaf of   differential endomorphisms
  of the left $\ms{O}_Y$-, right $f^{-1}\ms{D}$-module 
  $\ms{D}_{Y\to X}=f^*\ms{D}$. 
   This sheaf of operators $\ms{D}^f$ is itself a tdo on $Y$. 
  In the trivial example, we have $\ms{D}=\ms{D}_X$ and
   $\ms{D}^f=\ms{D}_Y$ for any $f$.  For maps $f:X\to Y$ and $g: Y\to Z$ and a tdo
   $\ms{D}$ on $Z$, we have
	$(\ms{D}^g)^f \simeq \ms{D}^{g\circ f}$.

%%%%%%%%   
\subsection{Inverse Image}
%%%%%%%%
%%%%%%

Let $f: Y\to X$ and $\ms{D}$ be as in the above section.  We denote the inverse image functor from 
$ \cl{M}(\ms{D})$ to $\cl{M}(\ms{D}^f)$ by $f^\circ$.  It is  defined as
\begin{equation*}
   f^\circ(\text{ - }) := \ms{D}_{Y\to X} \otimes_{f^{-1}\ms{D}}f^{-1}(\text{ - }).
\end{equation*}
Here $f^{-1}$ is the usual sheaf inverse image.  
The functor $f^\circ$ is right exact, exact when $f$ is flat, and 
     has finite left cohomological dimension. 
 
The category $\cl{M}(\ms{D})$ has enough projectives, 
and so the derived inverse image functor
\begin{equation*}
	\rm{\bf L}f^\circ: \rm{D^b}(\ms{D})\to \rm{D^b}(\ms{D}^f)
\end{equation*}
exists.   In \cite{b},  Borel defines the functor
\begin{equation*}
   f^!:= \rm{\bf L}f^\circ[d_{Y/X}]: \rm{D^b}(\ms{D})\to \rm{D^b}(\ms{D}^f),
\end{equation*}
 where  $d_{Y/X}= \dim{Y}-\dim{X}$.  Introducing the shift by $d_{Y/X}$ guarantees the 
 functor $f^!$ behaves well with respect to Verdier duality.

%%%%%%%
\subsection{Direct Image}
%%%%%%%

 Again let $f: Y\to X$ be as in \S \ref{tdos} and let $\ms{D}$ be a tdo on $X$.  We will
 define the direct image functor $f_+$, then examine this functor for $f$ a surjective submersion.
 The opposite sheaf $\ms{D}^\circ$ of any tdo $\ms{D}$
is again a tdo \cite[Prop. 11]{d-book}.  
 There is an \emph{isomorphism} of categories 
$	\cl{M}(\ms{D}^\circ)_{\rm{r}} \simeq  \cl{M}(\ms{D})$, which is the identity on objects.
Let $\omega_{Y/X}$ denote the relative canonical bundle for $f$.
 
 \begin{defn}
	Up to conjugation by the isomorphism $\cl{M}(\ms{D})\simeq \cl{M}
	(\ms{D}^\circ)_\rm{r}$, the direct image functor 
	$f_+: \rm{D}^b(\ms{D}^f) \to \rm{D}^b(\ms{D})$ is defined by 
	\begin{equation*}
		f_+(\text{ - }) = \rm{\bf R}f_*(\text{ - }\otimes\omega_{Y/X}\otimes_{(\ms{D}^\circ)^f}^{\rm{\bf L}}
		\ms{D}_{Y\to X}).
	\end{equation*}
\end{defn}

 \noindent This definition is the translation to left $\ms{D}$-modules of the usual construction: 
\begin{equation*}f_+: \rm{D^b}(\ms{D}^f)_{\rm{r}} \to \rm{D^b}(\ms{D})_\rm{r},
\quad   f_+(\text{ - }) = \rm{\bf R}f_*( \text{ - } \otimes^{\rm{\bf L}}_{\ms{D}^f} \ms{D}_{Y\to X}).
\end{equation*}

In general, the direct image $f_+$ is neither right nor left exact.   However, if $f$ is an affine morphism, then $f_*$ is exact and
       thus $f_+$ is right exact.  If $\ms{D}_{Y\to X}$ is a flat $\ms{D}^f$-module,
      such as when $f$ is an immersion, then the tensor product is exact, 
      so $f_+$ is left exact.  
      Putting these two special cases together we find that if $f$ an affine immersion,  then $f_+$ is exact.  
    Moreover, if $f$ is a closed immersion, then $f^!$ is the right adjoint to $f_+$.

          Let $f$ be a surjective submersion.   In this case,
          there is a locally free left $\ms{D}^f$-, right $f^{-1}\ms{O}_X$-module resolution 
          $\cl{T}_{Y/X}^\bul(\ms{D}^f)$ of $\ms{D}_{Y\to X}$ given by 
          \begin{equation*}
                \cl{T}_{Y/X}^{-k}(\ms{D}^f) = \ms{D}^f
                \otimes_{\ms{O}_Y}\wedge^k\ms{T}_{Y/X},\quad k\in\bb{\ZZ},
          \end{equation*}
   with the usual de Rham differential.  Here $\ms{T}_{Y/X}:= \Omega_{Y/X}^*$ is the sheaf 
   of local vector fields tangent to the fibers of $f$.  Note $\ms{T}_{Y/X}\subset \ms{D}^f$ since 
   the twist of $\ms{D}^f$ is trivial along these fibers.  
	The direct image with respect to this resolution gives
		\begin{equation*}\begin{array}{rcl}
		f_+(\ms{V}) & = & \rm{\bf R}f_*(\ms{V}\otimes\omega_{Y/X}\otimes_{(\ms{D}^\circ)^f}
		\cl{T}^\bul_{Y/X}(\ms{D}^\circ)^f) \\ 
		& = & \rm{\bf R}f_*(\Omega^\bul_{Y/X}(\ms{D}^f)
		\otimes_{\ms{D}^f}\ms{V})[d_{Y/X}]
		\end{array}
	\end{equation*}
	for all $\ms{V}\in\cl{M}(\ms{D}^f)$, where $\Omega_{Y/X}^\bul(\ms{D}^f)$ is the relative de Rham complex tensored with 
	$\ms{D}^f$.  In this case, it is transparent that $f_+[-d_{Y/X}] $ is the right adjoint of $f^\circ$.

   %%%%%
  \subsection{Homogeneous Twisted Sheaves of Differential Operators}
  %%%%%
  
In this section we classify homogeneous sheaves of twisted differential operators
 on generalized flag varieties.  The content  follows the analogous constructions in 
  \cite{d-book}. The generalized flag varieties are 
 homogeneous spaces $X$ for a complex reductive group $G$.  We consider only tdo's 
 which are equivariant with respect to the $G$-action on $X$; more precisely, 
 we will work exclusively with \emph{homogeneous} twisted sheaves of differential operators.
  \begin{defn} A \emph{homogeneous tdo} on a complex $G$-variety $X$ is a tdo $\ms{D}$  
  with a $G$-equivariant structure $\gamma$ and a morphism of algebras 
  $\alpha: \cl{U}(\fr{g})\to \G(X,\ms{D})$ satisfying:
	\paragraph{(H1)} The group $G$ acts on  $\ms{D}$ by algebra homomorphisms.
	\paragraph{(H2)} The differential of $\gamma$ agrees with the adjoint action --- that is,
		\begin{equation*}
			d\gamma_\xi(T) = [\alpha(\xi),T],\quad\forall\; \xi\in\fr{g},\; T\in\ms{D}.
		\end{equation*}
	\paragraph{(H3)} The map $\alpha$ is $G$-equivariant.  
  \end{defn}

  We now classify homogeneous tdo's  on  a generalized flag variety $X$ 
   of a complex reductive Lie group $G$.  Let $\fr{h}$ be the abstract Cartan for $\fr{g}$
  and let $\theta$ be the subset of simple positive roots corresponding to $X$.
  If $x\in X$ is any point and $\fr{p}_x$ the parabolic determined by $x$, define 
	\begin{equation*}
		\fr{h}_\theta = \fr{p}_x/[\fr{p}_x,\fr{p}_x].
	\end{equation*}  

  \begin{prop}
	The space $\fr{h}_\theta^*$ parameterizes isomorphism classes of homogeneous tdo's 
	on the partial flag variety $X$ of type $\theta$.
  \end{prop}
\noindent  This proposition is a special case of \cite[Theorem 1.2.4]{d-book}.
The proof is constructive; for completeness, we outline the construction of the 
homogeneous tdo $\ms{D}_{X,\lambda}$ 
for any $\lambda\in\fr{h}^*_\theta$.    Let $\fr{g}^\circ$ denote the trivial bundle 
$\ms{O}_X\otimes_\CC \fr{g}$. There is a surjection $\fr{g}^\circ\to\ms{T}_X$  with kernel 
$\fr{p}^\circ$, which has  geometric fiber $T_x\fr{p}^\circ = \fr{p}_x$ at $x\in X$,
  where $\fr{p}_x$ is the parabolic corresponding to $x$.  Let $P_x$ be the stabilizer of $x$ in $G$
  so that $P_x$ has $\fr{p}_x$ as its Lie algebra.  Any $\lambda\in \fr{p}_x^*$ which is 
  $P_x$-invariant determines a $G$-equivariant morphism 
  $\lambda^\circ: \fr{p}^\circ\to \ms{O}_X$.   In fact, these morphisms are in bijection with
  $P_x$-invariant linear forms on $\fr{p}_x$.   Define $\cl{U}^\circ :=
  \ms{O}_X\otimes_\CC\cl{U}(\fr{g})$ and the map $\phi_\lambda: \fr{p}^\circ\to \cl{U}^\circ$
   by  $\phi_\lambda( s)= s-\lambda^\circ(s)$ for  $s\in\fr{p}^\circ$.  The image
   of $\phi_\lambda$ generates a two-sided ideal $\ms{I}_\lambda$ in $\cl{U}^\circ$;
   finally, define 
  \begin{equation*}
  	\ms{D}_{X,\lambda} := \cl{U}^\circ/\ms{I}_{\lambda}.  
  \end{equation*} 
    The action of $G$ on $\cl{U}(\fr{g})$ induces an algebraic action on $\ms{D}_{X,\lambda}$,
  and similarly, the surjection $\cl{U}^\circ\to \ms{D}_{X,\lambda}$ determines a morphism
  \begin{equation*}
  	\alpha: \cl{U}(\fr{g})\to \G(X,\ms{D}_{X,\lambda})
  \end{equation*}
  upon taking global sections.  That the $G$-action and $\alpha$ satisfy $(H1)$--$(H3)$ is 
  obvious, and therefore, $\ms{D}_{X,\lambda}$ is a homogeneous twisted 
  sheaf of differential operators.

%%%%
\subsection{The Infinitesimal Character of $\ms{D}_{X,\lambda}$} 
%%%%

In this section we compute the infinitesimal character of $\G(X,\ms{D}_{X,\lambda})$. 
Let $[\lambda]\in \fr{h}^*/\cl{W}$ be the $\cl{W}$-orbit of $\lambda\in\fr{h}^*$.
Recall that when $X$ is the full flag variety, for any $\lambda\in\fr{h}^*$ there is an isomorphism
$\G(X,\ms{D}_{X,\lambda})\simeq \cl{U}_{[\lambda-\rho]}$ and all higher cohomology 
vanishes. 
 Consequently, we define $\ms{D}_\mu:= \ms{D}_{X,\mu+\rho}$ to compensate for  
 the $\rho$-shift in the infinitesimal character of global sections.   
  
Unfortunately, the global sections of $\ms{D}_{X,\lambda}$ for $X$ a partial flag variety 
do not always appear as a quotient of $\cl{U}(\fr{g})$.  However, we can determine
the infinitesimal character without computing global sections explicitly.  
Define 
  \begin{equation*}
  	\ms{D}_{\fr{h}_\theta} = \cl{U}^\circ/ [\fr{p}^\circ,\fr{p}^\circ]\cl{U}^\circ.
  \end{equation*}
  The quotient $\fr{h}_\theta^\circ = \fr{p}^\circ/D\fr{p}^\circ$ is the trivial bundle
  $\fr{h}_\theta^\circ = \ms{O}_X\otimes_\CC \fr{h}_\theta$.  
 For $\lambda\in\fr{h}_\theta^*$, the corresponding morphism 
 $\phi_\lambda: \fr{p}^\circ\to \cl{U}^\circ$ defining  $\ms{D}_{X, \lambda}$
 descends to $\phi_\lambda: \fr{h}_\theta^\circ\to \ms{D}_{\fr{h}_\theta}$.  
 The quotient $\fr{h}\to \fr{h}_\theta$ allows us to extend $\phi_\lambda$ to $\cl{U}(\fr{h})$,
 and then compose with the abstract Harish-Chandra isomorphism to get 
 a map  $\cl{Z}(\fr{g})\to \ms{D}_{\fr{h}_{\theta}}$.   Let $\rho_\theta$  and $\rho_n$
 be defined as in the introduction.   
 Note $\rho_\theta$ vanishes in the projection
  of $\fr{h}^*$ to the subspace $\fr{h}_\theta^*$.  Define
  \begin{equation*}
  	\ms{D}_\lambda = \ms{D}_{X, \lambda+\rho_n}.
  \end{equation*}
  Then, the  global sections of the tdo $\ms{D}_\lambda$
  has infinitesimal character $[\lambda-\rho_\theta]\in\fr{h}^*/\cl{W}$.
  
 We end with results illustrating some relationships between the twisting parameters
 for various homogeneous tdo's. 
  Let $p: X\to X_\theta$ be the projection of the full flag variety $X$ to the partial flag
  variety $X_\theta$ of type $\theta$.  There is then an equality
  $\ms{D}_{X, \lambda}^p = \ms{D}_{X,\lambda}$ and so
  \begin{equation*}\ms{D}_\lambda^p = \ms{D}_{\lambda-\rho_\theta}.\end{equation*}  
  Also, since the opposite tdo appears in the construction of the direct image, we include the 
  following proposition.   
   \begin{prop}
   	Let $\ms{D}_\lambda$ be any homogeneous tdo on the partial flag variety $X_\theta$. 
	Then, \begin{equation*}  \ms{D}_\lambda^\circ = \ms{D}_{-\lambda}.\end{equation*}
	Equivalently, we have $\ms{D}_{X_\theta, \lambda}^\circ = \ms{D}_{X_\theta, -\lambda+2\rho_n}$.  
   \end{prop}

  %%%%%
  \subsection{Anti-dominance and $\ms{D}$-affineness}
  %%%%%
  
  In this section we give some vanishing results for cohomology of $\ms{D}$-modules
  on generalized flag varieties.    
  Let $\fr{g}$ be a complex semi-simple Lie algebra, with abstract Cartan triple 
  $(\fr{h},\Sigma,\Sigma^+)$.   We will use $\Sigma^\vee$ to denote the co-roots in
  $\fr{h}$.    For $\lambda\in\fr{h}^*$, we say $\lambda$ is 
  \emph{anti-dominant} if $\alpha^\vee(\lambda)$ is not a positive integer for all 
  $\alpha\in\Sigma^+$.  Further, we say $\lambda$ is \emph{regular} if the 
  $\alpha^\vee(\lambda)$ are all non-zero as well.   
  If $\theta\subset \Pi^+$ is a subset of simple roots, let $\Sigma^+_\theta$ denote the 
  closure of $\theta$ in $\Sigma^+$ under addition.  Define $\Sigma_n:=
  \Sigma^+\backslash\Sigma^+_\theta$.  For $\fr{p}$ a parabolic of type $\theta$, 
  any specialization of $(\fr{h},\Sigma,\Sigma^+)$ to a Cartan triple for $\fr{p}$ will 
  send $\Sigma^+_\theta$ to positive roots contained in a Levi factor of $\fr{p}$ and 
  $\Sigma_n$ to the roots of the nilradical of $\fr{p}$.  Let $\rho_\theta$ 
  and $\rho_n$ be the half-sum of positive roots in $\Sigma^+_\theta$, respectively
  $\Sigma_n$.   Since $\fr{h}_\theta^*$ naturally embeds to a subspace of $\fr{h}^*$,
  we can define anti-dominance on $\fr{h}_\theta^*$ by restricting the condition on 
  $\fr{h}^*$.  However, it will be more useful to include a shift in the definition.
  \begin{defn}
  	The character $\lambda\in\fr{h}_\theta^*$ is \emph{anti-dominant}
	 if $\lambda-\rho_\theta\in\fr{h}^*$ is.  Likewise, $\lambda\in \fr{h}^*_\theta$ is
	 \emph{regular} if $\lambda-\rho_\theta\in\fr{h}^*$ is.  
  \end{defn}
  If $\theta$ is empty, $\fr{h}_\theta=\fr{h}$ and $\rho_\theta=0$, so this generalized definition is
  consistent with the original.    From \cite{bb}, we have the following definition and results.  
  
  \begin{defn}  Let $X$ be a generalized flag variety and $\ms{D}$ a tdo on $X$.  
  Say $X$ is \emph{$\ms{D}$-affine} if for every $\ms{F}\in\cl{M}(\ms{D})$ we have
   $\G(X,\ms{F})$  generated
  by global sections and $\rm{H}^i(X,\ms{F}) =0 $ for all $i>0$.
  \end{defn}
  
  \begin{prop} If $X$ is $\ms{D}$-affine, the global sections functor 
  \begin{equation*}
  	\G: \cl{M}(\ms{D})\to \cl{M}(\cl{D})
  \end{equation*}
  is an equivalence of categories, where $\cl{D}=\G(X,\ms{D})$.
  \end{prop} 
  
  Anti-dominance of $\lambda$ is necessary for $\ms{D}_\lambda$-affineness of 
  the full flag variety $X$; see for example \cite{m} or \cite{bb}.  Note our convention 
  of positive roots is the opposite of \cite{bb};  i.e., for them dominance rather than anti-dominance
  of $\lambda$ determines $\ms{D}_\lambda$-affineness.  

  \begin{thm}
  	Let $X$ be the full flag variety, $\lambda\in\fr{h}^*$.  
	\begin{enumerate}
		\item If $\lambda$ is dominant, then $\G: \cl{M}(\ms{D}_\lambda)\to
		\cl{M}(\cl{U}_{[\lambda]})$ is exact.  
		\item If $\lambda$ is also regular, then $\G$ is faithful.
	\end{enumerate}
	\end{thm}
	A consequence of this theorem is that for $\lambda$ anti-dominant and regular, 
		$\G$ gives an equivalence of 
		categories.  Its quasi-inverse $\Delta_\lambda$
		sends a $\cl{U}_{[\lambda]}$-module $V$ to 
		\begin{equation*}
			\Delta_\lambda(V) =\ms{D}_\lambda\otimes_{\cl{U}_{[\lambda]}}V.
		\end{equation*}   
		  
\noindent We prove the following proposition in \S \ref{embthm}.

  \begin{prop}
  	Let $\lambda\in \fr{h}_\theta^*$ be anti-dominant and regular.   Then $X_\theta$ is
	$\ms{D}_\lambda$-affine.
  \end{prop}

%%%%%%%%%%%%%%%%%%%%%%%%%%%%%%%%%%%%%%%%%%%%%%
%%%%%%%%%%%%%%%%%%%%%%%%%%%%%%%%%%%%%%%%%%%%%%

%%%%%%%
\section{The Equivariant Zuckerman Functor}
%%%%%%%

In this section, we recall the main definitions and some results
of the thesis of Pand\v zi\'c \cite{p}, including the construction of the equivariant Zuckerman
functor.  

%%%%%%
\subsection{$(\cl{A},K)$-Modules}
%%%%%%

  Let $(\cl{A}, K)$ be a pair consisting of an associative
algebra $\cl{A}$ over $\CC$ and $K$ a complex algebraic group.  The algebra $\cl{A}$ is
equipped with an algebraic $K$-action $\phi$, and a $K$-equivariant 
Lie algebra morphism $\psi: \fr{k}\to \cl{A}$ such that 
\begin{equation*}
    d\phi(\xi)(a) = [\psi(\xi),a],\quad \xi\in\fr{k}, \; a\in\cl{A}.
\end{equation*}  
Such pairs are called \emph{Harish-Chandra pairs}.
We will eventually take $\cl{A}$ to be global sections of a tdo on a generalized flag variety.  

\begin{defn}
  A \emph{weak $(\cl{A},K)$-module}  is a triple $(V,\pi,\nu)$ consisting of 
  \begin{enumerate}
     \item $V$ an $\cl{A}$-module with action $\pi$, and
     \item $V$ an algebraic $K$-module with action $\nu$, such that
     \item the $\cl{A}$-action map $\cl{A}\otimes V\to V$
     is $K$-equivariant.  In other words, \begin{equation*}\nu(k)\pi(a)\nu(k^{-1}) = \pi(\phi(k)a)\end{equation*}
     for all $k\in K$ and $a\in\cl{A}$. 
  \end{enumerate}
  An \emph{$(\cl{A},K)$-module} is a weak $(\cl{A},K)$-module
  such that 
  \begin{enumerate}
     \item[(4)] $d\nu = \pi\circ\psi$.
  \end{enumerate}
\end{defn}
\noindent This definition generalizes the notion of (weak) Harish-Chandra modules for the pair
 $(\fr{g},K)$.

Let $\cl{M}_{w}(\cl{A},K)$ be the category of all weak $(\cl{A},K)$-modules.  
{Morphisms} of weak $(\cl{A},K)$-modules are linear 
maps compatible with both the $\cl{A}$- and $K$-module 
structures.   Similarly, denote by $\cl{M}(\cl{A},K)$ the category of 
$(\cl{A},K)$-modules.  Let $\cl{C}(\cl{M}_{(w)}(\cl{A},K))$ and $\cl{K}(\cl{M}_{(w)}(\cl{A},K))$ 
denote the category of complexes and homotopy category of complexes of (weak) 
$(\cl{A},K)$-modules, respectively.  
The derived category $\rm{D}(\cl{M}_{(w)}(\cl{A},K))$ of (weak) $(\cl{A},K)$-modules is
constructed in the usual way, by localizing $\cl{K}(\cl{M}_{(w)}(\cl{A},K))$ with 
respect to quasi-isomorphisms.  Therefore for weak modules, we may simplify 
our notation by using $\cl{C}_{w}(\cl{A},K)$, etc.

%%%%%%%%%%
\subsection{Equivariant Derived Categories}\label{ForAK-cxs}
%%%%%%%%%%

Rather than working in the triangulated categories derived directly
from the abelian categories $\cl{M}(\cl{U}_{\chi},K)$ (for some  $\chi\in \fr{h}^*/\cl{W}$),   
for the purposes of 
localization it is necessary to work with the equivariant derived
category.  We give the needed definitions here.

\begin{defn}\label{equiHCcx}
   An \emph{equivariant $(\cl{A},K)$-complex} is a pair $(V^\bul,i)$ 
   with $V^\bul$ a complex of weak $(\cl{A},K)$-modules, and $i$ is 
   a linear map from $\fr{k}$ to graded linear degree $-1$ endomorphisms 
   of $V^\bul$ satisfying:
   \begin{enumerate}
      \item The $i_\xi$ are $\cl{A}$-morphisms for all $\xi\in\fr{k}$.
      \item The $i_\xi$ are $K$-equivariant for all $k\in K$.
      \item The sum $i_\xi i_\eta + i_\eta i_\xi = 0$ for all $\eta, \xi \in \fr{k}$.
      \item For every $\xi \in \fr{k}$, the sum $di_\xi + i_\xi d = \omega(\xi)$ where $\omega=\nu-\pi$.
   \end{enumerate}
\end{defn}
We summarize conditions $(1)$ and $(2)$ by stating 
$i\in \rm{Hom}_K(\fr{k},  \rm{Hom}_{\cl{A}}(V^\bul, V^\bul[-1] ))$,
where we take $\rm{Hom}_{\cl{A}}(V^\bul, V^\bul[-1] )$ in the category of graded 
$\cl{A}$-modules, and use the conjugation action of $K$.
Specifically, we have $K$ acting on  $f\in \rm{Hom}_{\cl{A}}(V^\bul, V^\bul[-1])$ by
\begin{equation*}
	(k.f)(v) = \nu(k).f(\nu(k^{-1})v)
\end{equation*} 
for all $k\in K$ and $v\in V^\bul$. 
The fourth condition implies the cohomology modules of $V^\bul$
are $(\cl{A},K)$-modules.

A morphism of equivariant $(\cl{A},K)$-complexes is a morphism
of complexes of weak $(\cl{A},K)$-modules which commutes with 
$i_\xi$ for all $\xi\in\fr{k}$. The category $\cl{C}(\cl{A},K)$
of equivariant $(\cl{A},K)$-complexes is abelian.
Two morphisms 
\begin{equation*}
	\phi,\psi: (V^\bul,i) \to (W^\bul,i)
\end{equation*} 
are {homotopic} if 
there exists a homotopy of complexes $h: V^\bul\to W^\bul[-1]$ which 
anti-commutes with $i_\xi$ for all $\xi\in\fr{k}$. That is, 
\begin{equation*}
	h\circ i_\xi = -i_\xi\circ h.
\end{equation*}  
Let $\cl{K}(\cl{A},K)$ 
be the homotopy category of equivariant $(\cl{A},K)$-complexes and 
$\rm{D}(\cl{A},K)$ its localization by quasi-isomorphisms.  
The category $\rm{D}(\cl{A},K)$ is known as the {equivariant derived category
of $(\cl{A},K)$-modules}.

For modules $V^\bul$ and $W^\bul\in \cl{C}_w(\cl{A},K)$,  define the homomorphism 
complex (\emph{Hom-complex}) by setting 
   \begin{equation*}
       \rm{Hom}^k(V^\bul, W^\bul) = \prod_p\rm{Hom}_{\cl{M}_{w}(\cl{A},K)}
       (V^p, W^{p+k}) 
   \end{equation*}
   with differential $d^k(f) = d_W\circ f -(-1)^k f\circ d_V$.  Clearly then
      $   \rm{Hom}^0(V^\bul, W^\bul)  =  \rm{Hom}_{\cl{C}_{w}(\cl{A},K)}(V^\bul, W^{\bul})$
       %\;\;\; \rm{and} \\
       and 
       $\rm{H}^0(\rm{Hom}^\bul(V^\bul, W^\bul))  = \rm{Hom}_{\cl{K}_{w}(\cl{A},K)}(V^\bul,W^\bul)$.
The Hom-complex for objects $(V^\bul, i)$ and $(W^\bul, j)$ in  $\cl{C}(\cl{A},K)$ is defined 
in the same way, but with morphisms $f\in \rm{Hom}^k(V^\bul,W^\bul)$ such that 
\begin{equation*}
	fi_\xi =(-1)^kj_\xi f,\quad \forall\; \xi\in\fr{k}.
\end{equation*}
Again, we have 
$         \rm{Hom}^0(V^\bul, W^\bul)  =  \rm{Hom}_{\cl{C}(\cl{A},K)}(V^\bul, W^{\bul})$
	and 
       $\rm{H}^0(\rm{Hom}^\bul(V^\bul, W^\bul))  =  \rm{Hom}_{\cl{K}(\cl{A},K)}(V^\bul,W^\bul)$.

Let $(\ast)=\cl{C},\cl{K}$ or $\rm{D}$.  There is a functor
 $     \rm{For_h}$ from  $(\ast)(\cl{A},K)$ to $(\ast)_{w}(\cl{A},K)$ which forgets the homotopy
 $i$ for the object $(V^\bul, i)$.  
Obviously  $\rm{For_h}: \cl{C}(\cl{A},K) \to \cl{C}_{w}(\cl{A},K)$ is 
faithful, but the same cannot necessarily be said for the homotopy or
derived categories.

For the pair $(\fr{g},K)$ and $\cl{A}=\cl{U}(\fr{g})$ (or more generally a quotient
of $\cl{U}(\fr{g})$), an important example of an equivariant complex in $\cl{C}^\rm{b}(\cl{A},K)$ 
is the \emph{standard complex} $\cl{N}(\fr{g})$ of $\fr{g}$.  The complex underlying
$\cl{N}(\fr{g})$ is the Koszul resolution of $\CC$ as a 
$\cl{U}(\fr{g})$-module.  That is,  for any integer $k$, we have 
\begin{equation*}
\cl{N}(\fr{g})^{-k} = \cl{U}(\fr{g})\otimes_\CC \wedge^{k}\fr{g}.
\end{equation*}
   If $u\otimes \tau\in\cl{N}(\fr{g})^{-(k+1)}$ with 
 $\tau = \tau_0\wedge\ldots\wedge\tau_k$,  
 the Koszul differential in degree $-(k+1)$ is
 \begin{equation*}\begin{array}{rcl}
    d^{-(k+1)}(u\otimes \tau) & = & 
    \sum_{i=0}^k(-1)^i u\tau_i\otimes \tau_0\wedge\ldots\hat{\tau_i}\ldots\wedge\tau_k
    \\ & + & \sum_{0\leq i<j\leq k} (-1)^{i+j}u\otimes[\tau_i,\tau_j]\wedge
    \tau_0\wedge\ldots\hat{\tau_i}\ldots\hat{\tau_j}\ldots\wedge\tau_k.
    \end{array}
 \end{equation*}
 The action $\pi_N$ of $\fr{g}$ in any degree is by left multiplication on $\cl{U}(\fr{g})$.  
 The action $\nu_N$ of $K$ is induced on each side of the tensor product
  by  the map $\phi: K\to \rm{Int}(\fr{g})$, and its differential 
 \begin{equation*}
    d\nu_N(\xi)(u\otimes\tau) = d\phi(\xi)u \otimes\tau + u\otimes d\phi(\xi)\tau.
 \end{equation*}
 There is a natural homotopy $i$ of these actions on $\cl{N}(\fr{g})$ given in the following proposition.  
  \begin{prop}
     For any $\xi\in\fr{k}$,  
 $u\otimes\tau\in\cl{N}(\fr{g})^{-k}$,  define
     $  i_\xi(u\otimes\tau) = -u\otimes \psi(\xi)\wedge\tau$.
 Then, we have $(\cl{N}(\fr{g}), i)\in\cl{C}^\rm{b}(\cl{A}, K)$.
  \end{prop}
\noindent The proof is a straightforward check, which we omit.

  %%%%%%%%%%%%
  \subsection{The Right Adjoint}
  %%%%%%%%%%%%
  The results of \cite{p} can be stated in terms of the right adjoint $\rm{Ind_h}$ to 
  $\rm{For_h}$ defined below.  For the geometric constructions of sections 4 and 5, 
  an alternative definition using tensor products will be more useful.  We give both 
  definitions and show they are equivalent.  
 
      Define    $\rm{Ind_h} : \cl{C}_w(\cl{A},K)\to \cl{C}(\cl{A}, K)$  to take     
	$V^\bul\in \cl{C}_{w}(\cl{A},K)$ to  $\rm{Ind_h}(V^\bul)=\rm{Hom}^\bul(\cl{N}(\fr{k}), V^\bul)$
	with $f\circ \pi_N(\xi) = \omega_V(\phi(\xi))\circ f$   for all  $f\in \rm{Ind_h}(V^\bul)$.
    The pair $(\cl{A},K)$ acts on a map $f: \cl{N}(\fr{k})\to V^\bul$ by 
    \begin{equation*}
    	\begin{array}{rcll}
		\pi(X)f & = & \pi_V(X)\circ f,&\forall\; X\in\cl{A}\\
		\nu(k)f & = & \nu_V(k)\circ f\circ \nu_N(k^{-1}), &\forall\; k\in K
	\end{array}
    \end{equation*}
    where $(\pi_V, \nu_V)$ denote the $(\cl{A},K)$-actions on $V^\bul$.  
    Then, for all $\xi\in\fr{k}$, 
    	$\omega(\xi)f = -f\circ \omega_N(\xi)$.
    There is a natural homotopy $i$ on $\rm{Ind_h}(V^\bul)$ given by  
$        i_\xi f = (-1)^{k-1} f \circ i_\xi$
   for $f$ a degree $k$ homomorphism, and any $\xi\in\fr{k}$.  The following proposition 
   can be proved by direct computation.  
   \begin{prop}
   	For any $V^\bul\in\cl{C}_{w}(\cl{A},K)$, the pair $(\rm{Ind_h}(V^\bul),i)\in\cl{C}(\cl{A},K)$.
   \end{prop}

\noindent   That $\rm{Ind_h}$ is right adjoint to the forgetful functor $\rm{For_h}$ is proved in detail in  \cite[Prop. 4.2.3]{k}.
  
  We can construct the right adjoint alternatively using a tensor product.  
  For  $V^\bul\in\cl{C}_{w}(\cl{A},K)$, let $\omega_V$ denote the $\fr{k}$-action
  on $V^\bul$ extended to $\cl{U}(\fr{k})$.  Then, $V^\bul$ can be made into a right $\cl{U}(\fr{k})$-module by 
  letting $u\in\cl{U}(\fr{k})$ act on $v\in V^\bul$ by 
  $
  	-\omega_V(u^\iota)v,
  $
  where $\iota$ is the principal anti-automorphism of $\cl{U}(\fr{k})$.  
 The total tensor product 
 $V^\bul\otimes_\fr{k}\cl{N}(\fr{k})$ has in degree $k$ the product
 \begin{equation*}
     (V^\bul\otimes_\fr{k}\cl{N}(\fr{k}))^k = \prod_{p\in\bb{Z}}  V^{p}\otimes_\fr{k}\cl{N}(\fr{k})^{k-p}
 \end{equation*}
  and its differential is the usual differential of a double complex 
  \begin{equation*}
       d^k(v\otimes n) = d_V^{p}v\otimes n+(-1)^pv\otimes d_N^{k-p}n,
  \end{equation*}
  for $v\otimes n\in V^{p}\otimes \cl{N}(\fr{k})^{k-p}$.  
  The algebra $\cl{A}$ acts on $V^\bul\otimes_\fr{k} \cl{N}(\fr{k})$ by
  \begin{equation*}
  	\pi(a)(v\otimes n) = \pi_V(a)v \otimes n,
  \end{equation*}
  for every $a\in \cl{A}$ and $v\otimes n\in V^\bul\otimes_\fr{k}\cl{N}(\fr{k})$.  
  The group $K$ acts diagonally as
  \begin{equation*}
  	\nu(k)(v\otimes n) = \nu_V(k)v \otimes \nu_N(k)n
  \end{equation*}
 for every $k\in K$.  Therefore, for every $\xi\in\fr{k}$ we have
 $	\omega(\xi)(v\otimes n) = 
	v\otimes \omega_N(\xi)n$.
  Define the homotopy of actions $i$ to be, up to a sign, the same as for $\cl{N}(\fr{k})$.  
  Specifically,  for $v\otimes n \in V^p\otimes_\fr{k}\cl{N}(\fr{k})^{k-p}$ and $\xi\in\fr{k}$, let
  $i_\xi (v\otimes n) = (-1)^{p} v\otimes i_\xi(n)$.
  
  \begin{prop}
  	With the above actions and $i$, the total tensor product $V^\bul\otimes_\fr{k}\cl{N}(\fr{k})$
	is an equivariant $(\cl{A},K)$-complex.  
  \end{prop}  
  
  On morphisms,  tensoring with $\cl{N}(\fr{k})$ sends $f$ to $f\otimes 1$
  for every $v\otimes n\in V^\bul\otimes \cl{N}(\fr{k})$.
  One can check this defines a chain morphism.
    The construction is the same if we replace $\cl{N}(\fr{k})$ by any equivariant 
  $(\cl{U}(\fr{k}),K)$-complex, so we have in fact proved the following general theorem.
  
  \begin{thm}
  	Let $V^\bul\in\cl{C}_{w}(\cl{A},K)$ and $W^\bul\in\cl{C}(\cl{U}(\fr{k}), K)$.
	Take the right $\cl{U}(\fr{k})$-module structure on $V^\bul$ determined by $\omega_V$.  
	Then $V^\bul\otimes_{\fr{k}}W^\bul$ is in $\cl{C}(\cl{A},K)$ with actions defined by
	\begin{equation*}
		\begin{array}{rcll}
			\pi(a)(v\otimes w) &=& \pi_V(a)v\otimes w & \forall\; a\in\cl{A} \;\;\text{and}\\
			\nu(k)(v\otimes w) & = & \nu_V(k)v\otimes \nu_W(k)w & \forall\; k\in K
		\end{array}
	\end{equation*}
	for all $v\otimes w\in V^p\otimes W^{k-p}$ and homotopies $i_\xi(v\otimes w) = 
	(-1)^{p-1}v\otimes i_\xi w$ for all $\xi\in\fr{k}$.  
  \end{thm}

\noindent Define $\rm{Ind'_h}(\text{ - }) := \text{ - }\otimes_{\cl{U}(\fr{k})}\cl{N}(\fr{k})[-d_K]$ as a functor from
$\cl{C}_{w}(\cl{A},K)$ to $\cl{C}(\cl{A},K)$. 
 \begin{prop}
     The functor $\rm{Ind'_h}$ is naturally isomorphic to $\rm{Ind_h}$. 
  \end{prop}
\begin{proof}
	
	Since $\fr{k}$ is reductive,  $\wedge^{d_K}\fr{k} = \CC$, the trivial representation of $\fr{k}$. 
	Therefore, the natural pairing
	\begin{equation*}
		\wedge^p\fr{k} \times \wedge^{d_K-p}\fr{k}\to \wedge^{d_K}\fr{k}
	\end{equation*}
	determines an isomorphism $ \wedge^{d_K-p}\fr{k}\tilde{\to}(\wedge^p\fr{k})^*$
	of $\fr{k}$-modules.  For $\tau\in \wedge^{d_K-p}\fr{k}$, let 
	$\tau^*\in(\wedge^p\fr{k})^*$ denote the linear map $\tau^*(\text{ - }) 
	= \text{ - }\wedge\tau$. Then, for any $k\in \ZZ$, this extends to a vector space isomorphism
	\begin{equation*}
		\Phi^k: (V^\bul\otimes \cl{N}(\fr{k}))^{k-d_K} \tilde{\to}
	 	\rm{Hom}^k(\cl{N}(\fr{k}),V^\bul).
	\end{equation*}
	Clearly, the $(\cl{A},K)$-module structure commutes with the vector space isomorphism.  
	Note also, since $(\xi\wedge\tau)^* = (-1)^{p+1}\tau^*\circ i_\xi$,
	the homotopy $i_\xi$ commutes with the isomorphism for all $\xi\in \fr{k}$.  
	That is, for $v\otimes\tau\in V^{k-p}\otimes \wedge^{d_K-p}\fr{k}$, we have
	\begin{equation*}
		\Phi^k(	i_\xi(v\otimes\tau) ) = i_\xi(v\otimes \tau^*).
	\end{equation*} 

	The  isomorphism $\Phi$ also commutes with differentials.
	Observe for all $(v_{k-p}\otimes\tau_{d_K-p})\in \prod_p V^{k-p}\otimes\wedge^{d_K-p}\fr{k}$,
	the compositions $d^k\Phi^k (v_{k-p}\otimes\tau_{d_K-p}) $ and 
	$\Phi^{k+1}d^{k-d_K}(v_{k-p}\otimes \tau_{d_K-p}) $ are equal if and only if 
	\begin{equation*}
		(-1)^p (d_N^{d_K-(p-1)}\tau_{d_K-(p-1)})^* = \tau_{d_K-(p-1)}^*\circ d_N^p
	\end{equation*}
	for all $p$.  Expand $\tau = \tau_{d_K-(p-1)} = \xi_1\wedge\ldots\wedge\xi_{d_K-(p-1)}$ and 
	let $\tau^\vee = \zeta_1\wedge\ldots\wedge\zeta_{p-1}$ be its complement. 
	For $\zeta\wedge d\tau$ and $d\zeta\wedge\tau$ to be nonzero, (up to sign)
	either $\zeta = \tau^\vee \wedge \xi_i$ for some $1\leq i\leq d_K-(p-1)$ or 
	$\zeta =\zeta_1\wedge\ldots\hat{\zeta_s}\ldots\wedge\zeta_{p-1}
	\wedge \xi_i\wedge\xi_j$ for some 
	$1\leq i<j\leq d_K-(p-1)$ and $1\leq s\leq p-1$.  One can verify explicitly that in either
	of these cases, we do in fact have the equality 
	\begin{equation*}
		(-1)^p (d_N^{d_K-(p-1)}\tau)^*(\zeta) = \tau^* (d_N^p(\zeta)).
	\end{equation*}
	Therefore, for any $V^\bul\in\cl{C}_{w}(\cl{A},K)$, there is an isomorphism of  
	 equivariant $(\cl{A},K)$-complexes
$		\rm{Hom}^\bul(\cl{N}(\fr{k}),V^\bul)\simeq V^\bul\otimes_\fr{k}\cl{N}(\fr{k})[-d_K]$
	which is transparently functorial in $V^\bul$. 
	  \end{proof}

%%%%
\subsection{Equivariant Zuckerman Functor}
%%%%

For $T\subset K$ a closed subgroup, there is a restriction functor 
$\rm{Res}_T^K$ from  $\rm{D}(\cl{A}, K)$ to $\rm{D}(\cl{A}, T)$,
given simply by restricting  the $K$-action on any object to $T$.  
Pand\v zi\'c gives an explicit construction of the right adjoint to this functor 
in \cite{p}, which he calls the \emph{equivariant Zuckerman functor} and denotes by 
$\G^\rm{equi}_{K,T}$.  We recall the construction here.  

Take the standard complex $\cl{N}(\fr{k})$ as an object of $\cl{C}(\fr{k},T)$ via
the restriction functor.  Let $R(K)$ be the ring of regular functions of $K$.  Then,
for any $V^\bul\in\cl{C}(\cl{A},T)$, we have 
\begin{equation*}
	R(K)\otimes_\CC V^\bul\in\cl{C}(\fr{k},T)
\end{equation*}
with the $(\fr{k},T)$-actions, denoted by $(\lambda_\fr{k}, \lambda_T)$ respectively. 
These actions are defined  for all  $k\in K$ and  $F\in R(K)\otimes_\CC V^\bul$ by 
      \begin{equation*}
         \begin{array}{rcll}
            (\lambda_\fr{k}(\xi)F )(k)&=& \pi_V(\xi) F(k) + L_\xi F (k) &\forall \xi\in\fr{k}, \\
           ( \lambda_T(t)F)(k) &=& \nu_V(t)F(kt) & \forall t\in T,\\
         \end{array}
      \end{equation*}
      where $(\pi_V,\nu_V)$ denote the $(\fr{k},T)$-actions on $V^\bul$, and the 
      $\fr{k}$-action is extended to $\cl{U}(\fr{k})$.  The homotopy $i$ is that for $V^\bul$.
          There is a commuting $(\cl{A},K)$-action on $R(K)\otimes_\CC V^\bul$, denoted
      by $(\pi_\G,\nu_\G)$ and defined for all $F\in R(K)\otimes V^\bul$ and $k\in K$ by
      \begin{equation*}
         \begin{array}{rcll}
            (\pi_\G(a)F)(k) &=& \pi_V(\phi(k)a)F(k), &\forall a\in \cl{A} \\
           (\nu_\G(h)F)(k) &=& F(h^{-1}k) & \forall h\in K.\\
            \end{array}
      \end{equation*}

Define the equivariant Zuckerman functor on an object $V^\bul\in\cl{C}(\cl{A},T)$ to be 
\begin{equation*}
    \G^\rm{equi}_{K,T}(V^\bul) = 
    \rm{Hom}^\bul(\cl{N}(\fr{k}), R(K)\otimes V^\bul)^{(\fr{k},T)},
\end{equation*}
with the Hom-complex taken in $\cl{C}_{w}(\cl{A},K)$, then taking $(\fr{k},T)$-invariants.  
      The $(\cl{A},K)$-action on 
      $\G^\rm{equi}_{K,T}(V^\bul)$ is denoted
      by $(\pi,\nu)$ and defined for all $f\in \G^\rm{equi}_{K,T}(V^\bul)$, 
      $n\in\cl{N}(\fr{k})$, and $k\in K$ as
      \begin{equation}
         \begin{array}{rcll}
            (\pi(a)f)(n)(k) &=&( \pi_\G(a)f(n))(k)=\pi_V(\phi(k)a)f(n)(k), &\forall a\in \cl{A} \\
           (\nu(h)f)(n)(k) &=& (\nu_\G(h)f(\nu_N(h^{-1})n))(k) = f(\nu_N(h^{-1})n)(h^{-1}k) & \forall k'\in K.
         \end{array}
      \end{equation}
      The homotopy $i_\xi$ acts on a morphism $f$ in degree $\ell$ by
      \begin{equation*}
      	(i_\xi f)(n)(k) = (-1)^{\ell+1} f(i_\xi n)(k)
      \end{equation*}
      for every $n\in \cl{N}(\fr{k})$, as in the definition of $\rm{Ind_h}$.

%%%%%%%%%%%%%%%%%%%%%%%%%%%%%%%%%%%%%%%%%%%%%
%%%%%%%%%%%%%%%%%%%%%%%%%%%%%%%%%%%%%%%%%%%%%

%%%%%%
\section{Equivariant Harish-Chandra Sheaves}
%%%%%%

The main technical construction required for the proof of Theorem \ref{mainthm}.\
is that of the \emph{geometric Zuckerman functor}.  This is the localization of the 
equivariant Zuckerman functor to the derived equivariant $\ms{D}$-module categories on
generalized flag varieties.  In this section, we define our categories of interest and construct the
geometric Zuckerman functor from the basic $\ms{D}$-module functors of \S 2.

   The category of (left) $G$-equivariant $\ms{O}_X$-modules is denoted 
   by $\cl{M}_G(X)$. The following theorem is well known:
   
   \begin{thm}\label{equieq}
	   If $G$ acts on $X$ freely, there is an equivalence of categories
	   $\cl{M}_G(X) \simeq \cl{M}(X/G).$
   \end{thm}
\noindent For homogeneous spaces, we can make a stronger statement.  Let $B$ be a complex linear
group, let $\cl{R}ep(B)$ be the category of algebraic representations of $B$.
	\begin{thm}
		If $X=G/B$, there is an equivalence of categories
	$		\cl{M}_G(X)\simeq \cl{R}ep(B).$
	\end{thm}

%%%%%
\subsection{Group Actions on Sheaves}
%%%%%

	Let $e_k: X\to K\times X$ be the map sending $x\mapsto (k,x)$.  Then
	if $\ms{F}$ is $K$-equivariant with structure isomorphism $\phi: \mu^*\ms{F}\tilde{\to}
	\pi^*\ms{F}$, pulling back along $e_k$ induces an isomorphism
$		e_k^*(\phi): s_k^*\ms{F}\tilde{\to} \ms{F},$
	where $s_k$ is the automorphism of $X$ given by $x\mapsto kx$.  
	In this way, we map $k\in K$ to the isomorphism
	\begin{equation*}
		e_k^*(\phi)\in \prod_{k\in K}\rm{Isom}(s_k^*(\ms{F}),\ms{F}),
	\end{equation*}
	and thus obtain an action of $K$ on sections.  For a local section 
	 $f\in\ms{F}$, define the action of $k\in K$ on $f$ to be
	the local section of $\ms{F}$ determined by 
$		\nu(k)f := e_k^*(\phi)^{-1}(f).  $
	One differentiates this action to obtain a corresponding Lie algebra action of $\fr{k}$ 
	on the sheaf $\ms{F}$.  The general construction for Lie algebra actions on sheaves
	is given in  \cite{bbjantzen}.

%%%%%%%
\subsection{Harish-Chandra Sheaves}
%%%%%%%

Let $(\fr{g},K)$ be a Harish-Chandra pair, let $X$ be a generalized flag variety for
 $\fr{g}$, and let $\ms{D}_\lambda$ be a homogeneous twisted sheaf 
of differential operators on $X$.  
   A \emph{weak Harish-Chandra sheaf} for the pair $(\ms{D}_\lambda, K)$ 
   is a quasi-coherent $\ms{D}_\lambda$-module $\ms{V}$  with a $K$-equivariant 
   $\ms{O}_X$-module structure such that the action of 
   $\ms{D}_\lambda$ is  $K$-equivariant.  
   A weak Harish-Chandra sheaf is a \emph{Harish-Chandra sheaf}
   if additionally the differential of the $K$-action on $\ms{V}$
      agrees with the action of $\fr{k}$ induced by $\ms{D}_\lambda$.  

A morphism of weak Harish-Chandra sheaves is a $\ms{D}_\lambda$-module
homomorphism which respects the underlying $K$-equivariant structure.   
As with weakly equivariant Harish-Chandra modules, we will use $\cl{M}_{w}(\ms{D}_\lambda,K)$
to denote the category of weak Harish-Chandra sheaves and $\cl{M}(\ms{D}_\lambda, K)$
for the category of Harish-Chandra sheaves.  
There is an equivalence of categories for $\lambda$ anti-dominant and regular 
\begin{equation*}
    \xymatrix{
        \cl{M}_{(w)}(\ms{D}_\lambda,K)\ar@<.5ex>[r]^>>>>>{\G} & 
        \cl{M}_{(w)}(\G(X,\ms{D}_\lambda), K). \ar@<.5ex>[l]^>>>>>>{\Delta_\lambda}
    }
\end{equation*}
We construct the derived equivariant Harish-Chandra
 sheaf category in the same way as the derived equivariant Harish-Chandra module category.  
\begin{defn}
   An \emph{equivariant Harish-Chandra sheaf} is a pair $(\ms{V}^\bul,i)$ 
   with $\ms{V}^\bul$ a complex of weak Harish-Chandra sheaves, and $i$ 
   a linear map from $\fr{k}$ to graded linear degree $-1$ endomorphisms 
   of $\ms{V}^\bul$ satisfying:
   \begin{enumerate}
      \item The $i_\xi$ are $\ms{D}_\lambda$-morphisms for all $\xi\in\fr{k}$.
      \item The map $i: \fr{k}\to \rm{Hom}_{\ms{D}_\lambda}(\ms{V}^\bul,\ms{V}^\bul[-1])$  
      is $K$-equivariant; that is, for all $k\in K$,
      \begin{equation*}
      	\nu(k)\circ i_\xi\circ \nu(k^{-1}) = i_{\rm{Ad}(k)\xi}.
      \end{equation*}
      \item For all $\eta,\xi\in\fr{k}$, the sum $i_\xi i_\eta + i_\eta i_\xi$ vanishes.
      \item There is the equality $di_\xi + i_\xi d = \omega(\xi)$, where $\omega=\nu-\pi$ and $\pi$ 
      is the action of $\fr{k}$ induced from $\ms{D}_\lambda$.   
   \end{enumerate}
\end{defn}
Define the Hom-complex $\rm{Hom}^\bul(\ms{V}^\bul, \ms{W}^\bul)$ in the
same way as for equivariant Harish-Chandra complexes.  The category of equivariant 
Harish-Chandra sheaves $\cl{C}(\ms{D}_\lambda, K)$ has equivariant 
 Harish-Chandra sheaves as objects and  for any $\ms{V}^\bul$ and $\ms{W}^\bul$ the 
 morphisms between them are $\rm{Hom}^0(\ms{V}^\bul, \ms{W}^\bul)$.
The homotopy category of equivariant Harish-Chandra sheaves 
$\cl{K}(\ms{D}_\lambda, K)$ has the same objects, but the zeroth cohomology 
$   \rm{H}^0(\rm{Hom}^\bul(\ms{V}^\bul, \ms{W}^\bul))$ of the Hom-complex from $\ms{V}^\bul$ to 
$\ms{W}^\bul$.
The derived equivariant Harish-Chandra sheaf category  $\rm{D}(\ms{D}_\lambda, K)$ is the localization of $\cl{K}(\ms{D}_\lambda, K)$ with respect to 
quasi-isomorphisms.
According to \cite{mp}, when $X$ is the full flag variety 
and $\lambda$ is regular, the derived global sections 
functor 
\begin{equation*}
	\rm{\bf R}\G: \rm{D}^+(\ms{D}_\lambda, K) \to 
	\rm{D}^+(\cl{U}_{[\lambda]}, K),
\end{equation*} 
is an equivalence.   This equivariant form of Beilinson-Bernstein localization
indicates the equivariant Harish-Chandra sheaf categories are the 
appropriate geometric category in which to work in the context of the results of \cite{p}.  

The categories of interest in 
the construction of the geometric Zuckerman functor are as follows.
Given a generalized flag variety $X$ for a Harish-Chandra pair $(\fr{g}, K)$ and a 
homogeneous tdo $\ms{D}_\lambda$ on $X$, we have the categories:
\begin{enumerate}
   \item  The abelian category $\cl{M}_{(w)}(\ms{D}_\lambda,K)$ of (weak) 
   Harish-Chandra sheaves on $X$ for the pair $(\ms{D}_\lambda, K)$.
   \item 
   The category $\cl{C}(\cl{M}_{(w)}(\ms{D}_\lambda,K))$ of complexes of 
   (weak) Harish-Chandra sheaves on $X$, and its homotopy category
    $\cl{K}(\cl{M}_{(w)}(\ms{D}_\lambda,K))$ and derived category
   $\rm{D}(\cl{M}_{(w)}(\ms{D}_\lambda,K))$.
   For concision, in the case of weak Harish-Chandra sheaves 
   we may abbreviate the above notation by 
      \begin{equation*}
          (\ast)_{w}(\ms{D}_\lambda,K) = (\ast)(\cl{M}_{w}(\ms{D}_\lambda,K))
      \end{equation*}     
      for $(\ast) = \cl{C}$, $\cl{K}$, or $\rm{D}$.   
   \item 
   The category of equivariant Harish-Chandra sheaves for $(\ms{D}_\lambda, K)$, 
   denoted $\cl{C}(\ms{D}_\lambda, K)$ and the homotopy and derived categories
   $(\ast)(\ms{D}_\lambda, K)$ with $(\ast) = \cl{K}$ and $\rm{D}$, respectively.
    
   \item 
   For $S\subseteq K$ a closed subgroup, the categories 
   $(\ast)(\cl{M}(\ms{D}_\lambda, K), S^{(w)})$, where $(\ast) = \cl{C}, \cl{K}$ or $\rm{D}$, 
   with objects consisting of complexes of  Harish-Chandra modules
   for  $(\ms{D}_\lambda, K)$ that are (weakly) $S$-equivariant complexes.   
   With respect to the notation of this list, we have
      \begin{equation*}
         (\ast)(\cl{M}(\ms{D}_\lambda, K), S^w) = (\ast)(\cl{M}(\ms{D}_\lambda, K\times S^w)).
      \end{equation*}
\end{enumerate}

 %%%%%%%%
\subsection{The Geometric $\rm{Ind_h}$}
%%%%%%%%

The forgetful functor 
\begin{equation*}
	\rm{For_h}: \cl{C}(\ms{D}_\lambda, K) \to \cl{C}_{w}(\ms{D}_\lambda,K)
\end{equation*}
forgets the homotopy $i$, as did the forgetful functor for the
equivariant $(\cl{A},K)$-complexes of \S \ref{ForAK-cxs}.   

\begin{prop}
	The forgetful functor $\rm{For_h}$ has a right adjoint. 
\end{prop}

We will now construct a functor $\rm{Ind_h}: \cl{C}_{w}(\ms{D}_\lambda, K) \to \cl{C}(\ms{D}_\lambda,K)$ and show it is the right adjoint to $\rm{For_h}$ in Proposition
\ref{geoindadjprop}.
For an object $\ms{V}^\bul~\in~\cl{C}(\ms{D}_\lambda,K)$, put 
\begin{equation*}
   \rm{Ind_h}(\ms{V}^\bul)=\rm{Hom}^\bul_{\cl{U}(\fr{k})}(\cl{N}(\fr{k}), \ms{V}^\bul),
\end{equation*}
where $\cl{N}(\fr{k})$ is  the standard complex thought of as a constant
sheaf on $X$.  For any $\xi\in\fr{k}$ define
\begin{equation*}
	i_\xi: \rm{Ind_h}(\ms{V}^\bul)\to \rm{Ind_h}(\ms{V}^\bul)[-1],\quad
	i_\xi f(u\otimes\tau) = (-1)^{j}f(i_\xi(u\otimes\tau))
\end{equation*}
for  $f$ in degree $j$, where $i_\xi: \cl{N}(\fr{k})\to \cl{N}(\fr{k})[-1]$ is the 
map  
\begin{equation*}i_\xi(u\otimes\tau) = -u\otimes \xi\wedge\tau.\end{equation*} 
  As a $(\ms{D}_\lambda, K)$-module,
$\ms{D}_\lambda$ acts on $f\in \rm{Ind_h}(\ms{V}^\bul)$ by its action on $\ms{V}^\bul$, 
and $K$ acts by conjugating $f$ by the $K$-actions on each factor.   That is,
for all local sections $T\in\ms{D}_\lambda$, we have 
$(Tf)(u\otimes\tau)=Tf(u\otimes\tau)$ and 
\begin{equation*}
	(\nu(k)f)(u\otimes\tau) = \nu_V(k)f(\nu_N(k^{-1})u\otimes\tau),
\end{equation*}  
where $\nu_V$ is the $K$-action on local sections of $\ms{V}^\bul$ and $\nu_N$ the $K$-action
on $\cl{N}(\fr{k})$.  

\begin{prop} The pair $(\rm{Ind_h}(\ms{V}^\bul),i)$
is an object of $\cl{C}(\ms{D}_\lambda,K)$ for all $\ms{V}^\bul\in \cl{C}_{w}(\ms{D}_\lambda,K)$.
\end{prop} 

\noindent We omit the proof which consists entirely of computations following the definitions.  

\begin{prop}\label{geoindadjprop}  The functor $\rm{Ind_h}$ is right adjoint to $\rm{For_h}$.  \end{prop}

\noindent  We again omit the completely computational proof, which can be found in \cite[\S 5.3]{k}. A version of this proposition is also presented without proof  
in \cite[Prop. 2.13.2]{bl2}. We use the Hom-complex in the proof 
of adjointness, which allows us to conclude that $\rm{Ind_h}$ is also the right
adjoint to $\rm{For_h}$ for the homotopy category.  Moreover,  the functor $\rm{Ind_h}$ 
preserves $\cl{K}$-injective objects.  The proposition below follows immediately from 
Lemma \ref{enoughinj} below.

   \begin{prop}
    If the category $\cl{C}_{w}(\ms{D}_\lambda,K)$ has enough $\cl{K}$-injectives, then the categories 
    $(\ast)(\ms{D}_\lambda,K)$ for $(\ast)=\cl{C}, \cl{K}$ or $\rm{D}$ have enough $\cl{K}$-injectives.    
 \end{prop}

\begin{lem}\label{enoughinj}
   If $f: \ms{V}^\bul\to \ms{I}^\bul$ is a quasi-isomorphism of $\ms{V}^\bul$ with a $\cl{K}$-injective 
   $\ms{I}^\bul$, then $\phi_f:= \rm{Ind_h}(f)$ is also a quasi-isomorphism. 
\end{lem}
\proof
   Recall from \cite[Prop. 1.5]{s} that $\ms{I}^\bul$ is 
   $\cl{K}$-injective if and only if for all diagrams 
   \begin{equation*}
      \xymatrix{
         \ms{V}^\bul\ar[r]^\phi \ar[d]_s & \ms{I}^\bul \\ \ms{W}^\bul
      }
   \end{equation*}
   with $s$ a quasi-isomorphism, there exists a unique 
   morphism $g_\phi: \ms{W}^\bul\to \ms{I}^\bul$ such that $[g_\phi s]=[\phi]$
   in the homotopy category.  
   
   Consider  the diagram
   \begin{equation*}
      \xymatrix{\cl{N}(\fr{k}) \ar[r]^\psi \ar[d]_s & \ms{I}^\bul[k] \\ \CC}
   \end{equation*}
   for any $k$, where $s$ is the usual quasi-isomorphism.  Then, 
   the unique $g_\psi$ which exists such that $[g_\psi s]=[\psi]$
   must send $1$ to $\psi^0(1)$.  
   
   Next, let $[v]\in \rm{H}^k(\ms{V}^\bul)$ such that $[f(v)]=[\psi^0(1)]\in
   \rm{H}^k(\ms{W}^\bul)$. Since $f$ is a quasi-isomorphism, such $[v]$ is unique.
   Thus there exists a unique $g_f: \CC\to \ms{I}^\bul$ such that
   $[g_fs]=[\phi_f(v)]$.  We have 
 \begin{equation*}[g_f(1)]= [f(v)]=
 [g_\psi(1)]\in \rm{H}^k(\ms{W}).\end{equation*}
   This implies $[g_f]=[g_\psi]$, and consequently  
   $[\phi_f(v)]=   [\psi]$. 
   Moreover, this result 
   verifies both injectivity and surjectivity of the 
   morphism \begin{equation*}[\phi_f]: \rm{H}^k(\ms{V}^\bul)\to \rm{H}^k(\rm{Ind_h}(\ms{I}^\bul))\end{equation*}
   for all~$k$.   
   \endproof

%%%%%%%%%
\subsection{Reduction Principle}
%%%%%%%%%
 
Although the category $\rm{D}(\ms{D}_\lambda, K)$ is not derived from its heart
$\cl{M}(\ms{D}_\lambda, K)$, we define in this section a notion of a derived functor.  
   Trivially, any exact functor on $\cl{M}_{w}(\ms{D}_\lambda,K)$ extends to a
   functor on $\rm{D}(\ms{D}_\lambda,K)$.   Moreover, the forgetful functor from
   $\cl{M}(\ms{D}_\lambda, K)$ to $\cl{M}_{w}(\ms{D}_\lambda,K)$ and 
   $\cl{C}(\ms{D}_\lambda,K)$ to $\cl{C}_{w}(\ms{D}_\lambda,K)$ are 
   obviously faithful, so any functor on the weakly equivariant categories
   lifts immediately to Harish-Chandra sheaves and equivariant Harish-Chandra 
   complexes.  The lifting of properties such as exactness, adjointness, etc. follows trivially.
    The forgetful functors at the homotopy and derived level  are not necessarily
   faithful, since homotopies of morphisms in the equivariant 
   Harish-Chandra categories must anti-commute with the additional structure map $i$
   with which each object is equipped. In this case, the lifting of functors from the 
   weakly equivariant categories to the equivariant categories is non-trivial, but made
   possible by the existence of the right adjoint $\rm{Ind_h}$.  
   
  Since $\cl{C}_{w}(\ms{D}_\lambda,K)$ has enough $\cl{K}$-injectives, so does 
   $\cl{C}(\ms{D}_\lambda,K)$ and  
   the corresponding homotopy and derived 
   categories.  In this case, any left exact functor on
   $\cl{M}_{w}(\ms{D}_\lambda,K)$ 
   defines a right derived functor on $\rm{D}^*(\ms{D}_\lambda, K)$, up to imposing 
   appropriate finiteness conditions (depending on $*$) for cohomological dimension.
	 In the next sections, 
	the construction of the functors needed for the main theorem are given in terms of 
	left exact functors on $\cl{M}_{w}(\ms{D}_\lambda,K)$ with finite right cohomological 
	dimension.

%%%%%%%%%
\subsection{Restriction of Group Actions}
%%%%%%%%%

Let $S\subset K$ be a closed subgroup.  There is a restriction functor 
$\rm{Res}_S^K$ from $\cl{M}(\ms{D}_\lambda, K)$ to $\cl{M}(\ms{D}_\lambda,S)$ 
defined by restricting the $K$-action to $S$.  
Since the $K$-action on an equivariant sheaf $\ms{V}$ is defined by 
an isomorphism $\phi: \mu^*\ms{V}\tilde{\to}\pi^*\ms{V}$, with $\mu, \pi: K\times X\to X$
the usual action and projection morphisms (respectively),  the restriction 
of the $K$-action comes from taking $\phi$ to $j^*\phi$,
where $j: S\times X\to K\times X$ is the obvious inclusion.  
The restriction functor is exact and therefore extends to the equivariant derived categories.

\begin{defn}
   For $S\subset K$ a closed subgroup 
   define the restriction functor 
   \begin{equation*}
       \rm{Res}_S^K: \cl{C}(\ms{D}_\lambda, K)\to \cl{C}(\ms{D}_\lambda, S)
   \end{equation*}
   by restricting the $K$-action to $S$ and the map $i$ to $\fr{s}$.  
\end{defn}

\noindent Let $\pi_*^K$ be the direct image functor $\pi_*$ composed with the
functor of $K$-invariant sections $(\text{ - })^K$.  

\begin{prop}\label{adjtotriv}
	There is a natural isomorphism $\rm{Res}_S^K \simeq \pi_*^K\mu^*$.

\end{prop}
\proof
	It is enough to prove the proposition for $\cl{M}(\ms{D}_\lambda, K)$, since $\mu^*$ 
	and $\pi_*^K$ are exact.  	Let $\ms{V}\in\cl{M}_{w}(\ms{D}_\lambda,K)$. 
	There is an isomorphism $\phi: \mu^*\ms{V} \tilde{\to} \pi^*\ms{V}$ 
	defining the $K$-action.  Additionally, the functor $\pi_*^K$ is the inverse to 
	$\pi^*$ in the equivalence 
	\begin{equation*}\cl{M}(\ms{O}_X, S)\simeq \cl{M}(\ms{O}_{K\times X},K\times S).\end{equation*}
	Therefore, the equivalence $\phi$ pushes down to an isomorphism $\pi_*^K(\phi)$ from  
	$\pi_*^K\mu^*\ms{V}$ to  $\rm{Res}_S^K\ms{V}$.    Let 
	$\pi_S,\mu_S: S\times X\to X$ be the projection and action morphisms, and let $j: 
	S\times X\to K\times X$ be induced from the inclusion $S\to K$.  Then we have 
	 $\mu_S = \mu\circ j$ and likewise $\pi_S=\pi\circ j$.  Moreover, there is a shear 
	 morphism  $s: K\times X\to K\times X$ such that we have $\mu = \pi\circ s$.  
	 Consequently,we have
$		\pi_S^*\pi_*^K = j^*$. 
	Thus the isomorphisms $\pi_S^*\pi_*^K(\phi)$ and $j^*(\phi)$  from $\mu_S^*\ms{V}$
	to $\pi_S^*\ms{V}$ are equal.
		
	Now, let $f: \ms{V}\to \ms{W}$ be any morphism in $\cl{M}_{w}(\ms{D}_\lambda,K)$.  
	It restricts to a morphism in $\cl{M}(\ms{D}_\lambda, S^w)$.  Then, the equalities
$		f\circ \pi_*^K(\phi) = 	\pi_*^K(\pi^*(f)\circ \phi)
		= \pi_*^K(\phi)\circ \pi_*^K\mu^*(f)$
	complete the proof.
\endproof

%%%%%%%%%%
\subsection{The Geometric Zuckerman Functor}\label{gzfunctsec}
%%%%%%%%%%

As in the algebraic setting, the geometric restriction functor $\rm{Res}_S^K$ has a 
right adjoint.
We will construct the geometric Zuckerman functor $\G_{K,S}^{\rm{geo}}$ and 
prove that it is the right adjoint to $\rm{Res}_S^K$.

Let $X$ be a $K$-variety and $S$ a closed subgroup of $K$.  
Lemma 1.8.6 of \cite{bbjantzen} states there is an equivalence
\begin{equation*}\cl{M}_{w}(\ms{D}_\lambda,K) \simeq \cl{M}(\ms{D}_\lambda\otimes_\CC\cl{U}(\fr{k}), K).\end{equation*}
Also in \cite{bbjantzen}, given a free $K$-action on $K\times X$, let 
$q: K\times X\to X$ be the quotient map.  There is then a pair of equivalences
   \begin{equation}\label{weak}
      \xymatrix{
       \cl{M}_{w}(\ms{D}^q_\lambda,K) \ar@<.5ex>[r]^{q_*^K} & 
       \cl{M}(\ms{D}_\lambda\otimes\cl{U}(\fr{k})) \ar@<.5ex>[l]^{q^\circ}
       }\;\; \text{ and } \end{equation}
   
         \begin{equation}\label{strong}
         \xymatrix{
       \cl{M}(\ms{D}_\lambda^q, K) \ar@<.5ex>[r]^>>>>>{q_*^K} & 
       \cl{M}(\ms{D}_\lambda). \ar@<.5ex>[l]^<<<<<{q^\circ}
       }\end{equation}
\noindent  
The equivalence (\ref{weak}) generates an equivalence for weakly 
equivariant complexes 
\begin{equation*}
    \xymatrix{
        \cl{C}_{w}(\ms{D}^q_\lambda,K)\ar@<.5ex>[r]^{q_*^K} & 
        \cl{C}(\ms{D}^q_\lambda\otimes\cl{U}(\fr{k})). \ar@<.5ex>[l]^{q^\circ}
    }
\end{equation*}

We can naturally extend these equivalences to include (weakly) $S$-equivariant
sheaves whenever $q$ is a $S$-equivariant morphism and the $S$-action on
$K\times X$ commutes with the $K$-action. 
 Let $\pi$ and $\mu: K\times X\to X$ be the usual projection and action maps, respectively.  
The product $K\times S$ acts on $K\times X$ by 
$   (k',s)(k,x) = (k'ks^{-1},sx)$  for all $(k, x)\in K\times X$ and $(k',s)\in K\times S$.
There is also a $K\times S$-action on $X$ with $K$ acting trivially and $S$ acting by the 
restriction of the original $K$-action.  With these actions on $K\times X$ and $X$, the map 
$\pi$ is $K\times S$-equivariant.  Similarly, for the $K\times S$-action on $X$ given by the 
$\mu$-action of $K$ and the trivial $S$-action, the morphism $\mu$ is $K\times S$-equivariant.
If $S$ is trivial, we can recover the situation of \cite{bbjantzen} described above
 by letting $q=\pi$ or $\mu$.

The $S$-equivariance of $\pi$ and the fact that the
inverse image $\pi^\circ$ is an exact functor from $\cl{M}(\ms{D}_\lambda, S^{(w)})$ to 
$\cl{M}(\ms{D}_\lambda^\pi, K\times S^{(w)})$ imply that $\pi^\circ$ extends to a
functor of  derived equivariant categories
\begin{equation*}
   \pi^\circ: \rm{D}(\ms{D}_\lambda, S^{(w)})\to \rm{D}(\cl{M}(\ms{D}_\lambda^\pi, K), S^{(w)}).
\end{equation*}
Since $\ms{D}_\lambda$ is a $K$-equivariant $\ms{O}_X$-module on $X$, 
there is a canonical isomorphism between
	$\mu^*\ms{D}_\lambda$ and $\pi^*\ms{D}_\lambda$.
Consequently, we have an induced isomorphism 
$\ms{D}_\lambda^\mu\simeq \ms{D}_\lambda^\pi$ 
of their respective sheaves of differential endomorphisms.  Therefore, there is a natural
isomorphism of categories
   $\cl{M}(\ms{D}_\lambda^\mu) = \cl{M}(\ms{D}_\lambda^\pi)$.  

We use the above equivalences to motivate our construction of the geometric Zuckerman 
functor. The direct image $\mu_*^K$ above does not land in the category of strongly
 $K$-equivariant sheaves, since it is not clear what the $\ms{D}$-module structure is away
 from the $K$-equivariant sections.  The $\ms{D}$-module direct image $\mu_+$ 
 corrects for this problem, as we show in the proposition below. 
\begin{prop} The $\ms{D}$-module direct image functor $\mu_+$ takes
      $\cl{M}(\ms{D}_\lambda^\mu, K\times S^{(w)})$ to 
     $ \cl{C}^\rm{b}(\cl{M}(\ms{D}_\lambda, S^{(w)}), K)$ and
      $\cl{C}^\rm{b}(\cl{M}(\ms{D}_\lambda^\mu, K), S^{(w)})$ to 
      $\cl{C}^\rm{b}(\ms{D}_\lambda, K\times S^{(w)})$.
\end{prop}
\proof
	We need only construct a homotopy of the $\fr{k}$-actions
	\begin{equation*}
		i: \fr{k} \to \rm{Hom}_{\ms{D}_\lambda}(\mu_+\ms{V},\mu_+\ms{V}[-1])
	\end{equation*}
	for every object $\ms{V}\in \cl{M}(\ms{D}_\lambda^\mu, K\times S^{(w)})$.  
	Since $\mu$ is a surjective submersion, the direct image functor $\mu_+$ is equal to
$\mu_* ( \text{ - }\otimes_{\ms{O}_{K\times X}}\cl{T}^\bul_{K\times X/X})$.
	The sheaf $\cl{T}_{K\times X} $ equals the exterior tensor product 
	$\cl{T}_K\boxtimes \cl{T}_X$, and therefore the relative sheaf of differentials
	\begin{equation*}
		\cl{T}_{K\times X/X} \simeq \pi_K^*\cl{T}_K.
	\end{equation*}
	Moreover, $K$ is affine so $\cl{T}_K = \ms{O}_K\otimes_{\CC} \fr{k}$ and hence
	the inverse image is 
	$\pi_K^*\cl{T}_K^\bul= \ms{O}_{K\times X}\otimes_{\CC} \wedge^\bul\fr{k}$.
	Namely,  we have equalities
	\begin{equation*}
		\mu_+(\text{ - }) = \mu_*(\text{ - }\otimes_{\CC}\wedge^\bul\fr{k}) = 
		\mu_*(\text{ - })\otimes_{\cl{U}(\fr{k})}\cl{N}(\fr{k}).
	\end{equation*}
	Define the homotopy map $i$ to be that coming from $\cl{N}(\fr{k})$.  
	
\endproof

Define the functor of $K$-invariant sections on $\cl{C}_{(w)}(\ms{D}_\lambda,K)$ by
\begin{equation*}
    (\text{ - })^K = (\text{ - })^{(K, \cl{N})}: \cl{C}_{(w)}(\ms{D}_\lambda,K) \to \cl{C}(\ms{D}_\lambda),
\end{equation*}
where $\cl{N}=\cl{U}(\fr{k})$ in the case of weakly equivariant complexes, and $\cl{N}(\fr{k})$
for equivariant complexes. By $\cl{N}(\fr{k})$-invariants,
we mean $\cl{U}(\fr{k})$-invariants which are also invariant for the homotopy map $i$.   
   The invariants functor $(\text{ - })^K$ is right adjoint to the trivial inclusion 
   	$\rm{Triv}_{(w)}$ from  $\cl{C}(\ms{D}_\lambda)$ into $\cl{C}_{(w)}(\ms{D}_\lambda,K)$.

\begin{defn}
   For $(\ast)=\cl{C}$ or $\cl{K}$, define the geometric Zuckerman functor from  
   $(\ast)^{\rm{b}}(\ms{D}_\lambda, S)$ to $(\ast)^{\rm{b}}(\ms{D}_\lambda, K)$ to be
   \begin{equation*}
       \G_{K, S}^\rm{geo} : = \mu_+^S\pi^\circ [-d_K].
   \end{equation*}
\end{defn}
\noindent By the reduction principle, the geometric Zuckerman functor   $\G_{K, S}^\rm{geo}$ is also defined for derived equivariant categories.

\begin{prop}
   The geometric Zuckerman functor $\G_{K, S}^\rm{geo}$ is right adjoint to $\rm{Res}_S^K$.  
\end{prop}
\proof
	Recall there is a natural isomorphism
	\begin{equation*}\rm{Res}_S^K \simeq \pi_*^K\mu^\circ \end{equation*}
	and note $\pi_*^K$ is the inverse to $\pi^\circ$.    Therefore, we need only show
	$\mu^\circ$ is left adjoint to  $ \mu_+^S[-d_K]$.
	Since $\mu$ is smooth, we know $\mu^\circ \dashv \mu_+^S[-d_K]$ is 
	an adjoint pair when $S$ is trivial.  
	Additionally, Proposition \ref{adjtotriv}.\ shows that
	$(\text{ - })^S$ is right adjoint to $\rm{Triv}_{(w)}$.  
	The restriction functor $\rm{Res}_S^K$ factors through $\cl{C}(\ms{D}_\lambda, K\times S)$
	by the functor $\rm{Triv}$.  In fact, we should have defined $\rm{Res}_S^K~=~\pi_*^K\mu^\circ\circ \rm{Triv}$,
	in which case it is clear that
	\begin{equation*}
		\mu^\circ\circ \rm{Triv} \dashv \mu_+^S[-d_K]
	\end{equation*}
	as a composition of two adjoint pairs.  
\endproof

Before moving on to examine properties of $\G_{K,S}^\rm{geo}$, it will be useful 
to generalize the equivalences (\ref{weak}) and (\ref{strong}) to derived equivalences for 
equivariant complexes.

\begin{lem}\label{equivariance}
   For $K$ acting freely on $Z$ and $\ms{D}$ a sheaf of twisted differential operators on 
   the quotient $Z/K$, there is an equivalence of categories
   \begin{equation*}
      \xymatrix@C=4pc{
          \cl{C}(\ms{D}^q, K) \ar@<.5ex>[r]^<<<<<<<<<{q_+^K[-d_K]} & \cl{C}(\ms{D}). \ar@<.5ex>[l]^>>>>>>>>>{q^\circ}
          }
   \end{equation*}
\end{lem}
\proof
	By the reduction principle, the lemma follows from the equivalence
   \begin{equation*}
      \xymatrix{
          \cl{M}_{w}(\ms{D}^q,K) \ar@<.5ex>[r]^>>>>>{q_*^K} & \cl{M}(\ms{D}) \ar@<.5ex>[l]^>>>>>{q^\circ}
          }
   \end{equation*}
   whenever $q_+^K[-d_K] \simeq q_*^K$.
 Recall there are isomorphisms
 $q_+^K(\ms{V})[-d_K] \simeq (q_*(\ms{V})\otimes_{\cl{U}(\fr{k})}\cl{N}(\fr{k}))^K[-d_K]$
 and 
$  (q_*(\ms{V})\otimes_{\cl{U}(\fr{k})}\cl{N}(\fr{k}))^K[-d_K]
= \rm{Hom}_\fr{k}(\cl{N}(\fr{k}),q_*\ms{V})^K 
= q_*^K(\ms{V})$ from which the desired equivalence of abelian categories follows.
\endproof
\noindent This equivalence respects (weak) $S$-equivariance when the $S$-action commutes
with $K$ and $q$ is $S$-equivariant.

%%%%%%%%
\subsection{Properties of $\G_{K,S}^\rm{geo}$}
%%%%%%%%

In this section, we use the geometric Zuckerman functor $\G_{K,S}^\rm{geo}$
to construct standard modules on generalized flag varieties.  
We fix the following notations.  As above, let $(\fr{g},K)$ be a Harish-Chandra pair, 
let $X$ be a generalized flag variety of $\fr{g}$, let $Q$ be a $K$-orbit on $X$ and let $i: Q\to X$ be the inclusion.
Fix a point $x\in Q$ and let $i_x: x\to Q$ and $j_x=i\circ i_x$ denote the inclusions of $x$ to $Q$
and $X$ respectively.  Fix a homogeneous tdo $\ms{D}_\lambda$ on $X$ and let 
$\cl{D}_{[\lambda]}$ denote the global sections
of $\ms{D}_\lambda$, with $[\lambda]\in\fr{h}^*/\cl{W}$ the Weyl group orbit of 
$\lambda$.  Denote the stabilizer of $x$ in $K$ by $S_x$.

\begin{prop} The diagram below is commutative:
   \begin{equation*}
       \xymatrix{
          \rm{D^b}(\ms{D}_\lambda^i, S_x) \ar[r]^{\G_{K, S_x}^\rm{geo}} \ar[d]_{i_+} &
          \rm{D^b}(\ms{D}_\lambda^i, K) \ar[d]^{i_+} \\
          \rm{D^b}(\ms{D}_\lambda, S_x) \ar[r]^{\G_{K, S_x}^\rm{geo}} & \rm{D^b}(\ms{D}_\lambda, K).
       }
   \end{equation*}
\end{prop}
\proof
	Consider the diagram
		\begin{equation*}
			\xymatrix{
				K\times Q \ar[r]^i \ar[d]_\mu & K\times X \ar[d]^\mu	 \\
				Q\ar[r]^i & X.
			}
		\end{equation*}
		Then, $i_+\mu_+= \mu_+i_+$ and moreover, $i_+\mu_+^{S_x} = \mu_+^{S_x}i_+$
		since $i$ is $S_x$-equivariant.  Base change for $\ms{D}_\lambda$-modules allows us
		to commute $i_+$ past $\pi^\circ$.  
\endproof

\begin{prop}\label{equicommprop}
   The diagram below is commutative:
   \begin{equation*}
      \xymatrix{
                 \rm{D^b}(\ms{D}_\lambda, S_x) \ar[r]^{\G_{K, S_x}^\rm{geo}} \ar[d]_{\rm{\bf R}\G} &
          \rm{D^b}(\ms{D}_\lambda, K) \ar[d]^{\rm{\bf R}\G} \\
          \rm{D^b}(\cl{D}_{[\lambda]}, S_x) \ar[r]^{\G_{K, S_x}^\rm{equi}} 
          & \rm{D^b}(\cl{D}_{[\lambda]}, K).
      }
   \end{equation*}
 \end{prop}
\proof
	The reduction principle allows us to work with categories of complexes alone.  
	It is clear from the previous constructions that for any 
	$\ms{V}^\bul\in \cl{C}^\rm{b}(\ms{D}_\lambda, S_x)$, we have
	\begin{equation*}
		\rm{\bf R}\G(X,\G^\rm{geo}_{K,S_x}\ms{V}^\bul) =
		 \rm{Hom}_{(\fr{k}, S_x,\cl{N}(\fr{s}))}^\bul(\cl{N}(\fr{k}), 
		 \rm{\bf R}\G(X,\mu_*\pi^\circ\ms{V}^\bul))
	\end{equation*}
	as $(\cl{D}_{[\lambda]},K)$-complexes.  For any open set $U\subset X$, we have by 
	definition
	$\mu_*\pi^\circ\ms{V}^\bul(U) = R(K)\otimes \ms{V}^\bul(U)$ where $R(K)$ is the ring
	of regular functions on $K$.  Therefore, we have
	\begin{equation*}
		\rm{\bf R}\G(X,\mu_*\pi^\circ\ms{V}^\bul) = R(K)\otimes\rm{\bf R}\G(X,\ms{V}^\bul). 
	\end{equation*} 
	Local sections of $\mu_*\pi^*\ms{V}^\bul$ are
	functions $F$ from $K$ to $\ms{V}^\bul$.  
	
	The shift isomorphism $\sigma: K\times X\to K\times X$ (defined by $\sigma(k,x)=(k,kx)$) 
	relates $\mu$ and $\pi$ by $\mu = \pi\circ \sigma$. 
	Let $\phi: \mu^*\ms{D}_\lambda\tilde{\to}\pi^*\ms{D}_\lambda$ be the $K$-equivariance
	isomorphism, and let $\phi^*: \ms{D}_\lambda^\mu\tilde{\to}\ms{D}_\lambda^\pi$
	be the induced isomorphism of tdo's.   The isomorphism $\phi^*$ locally sends a 
	section $T\in\ms{D}^\mu_\lambda$ to $\phi\circ T\circ \phi^{-1}$,
	where $\phi: \mu^*\ms{D}_\lambda\tilde{\to} \pi^*\ms{D}_\lambda$ sends 
	$X\in \mu^*\ms{D}_\lambda$ to $\phi(X) = X\circ \sigma^{-1}$.  
	Now consider the $S_x$-action.  Let 
	\begin{equation*}
		\tilde{\mu},\tilde{\pi}: S_x\times K\times X\to K\times X
	\end{equation*}
	be the action and projection morphisms, respectively.  For any
	$t\in S_x$, let 
	\begin{equation*}
		\tilde{e}_t: K\times X\to S_x\times K\times X
	\end{equation*} 
	be the inclusion
	$(k,x)\mapsto (t,k,x)$.  If $\psi$ is the $S_x$-equivariance morphism for 
	$\ms{V}^\bul$, then the $S_x$-action on $\pi^*\ms{V}^\bul$ is given by 
	\begin{equation*}
		\lambda_s(t)(f\otimes v) = \pi^*e_t^*(\psi)^{-1}(f\otimes v),  
	\end{equation*}
	where $e_t: X\to S_x\times X$, $e_t: x\mapsto (t,x)$.  Hence for a local
	section $F\in \mu_*\pi^*\ms{V}^\bul$, we have 
	\begin{equation*}
		(\lambda_s(t)F)(k) = e_t^*(\psi)^{-1}(F(kt)).
	\end{equation*}

	The $\cl{U}(\fr{k})$-module structure on $\mu_*\pi^*\ms{V}^\bul$
	is induced by  $\phi^*$.   For every $F\in \mu_*\pi^*\ms{V} $ and $\xi\in\fr{k}$, we have
	\begin{equation*}
		(\lambda_\fr{k}(\xi)F)(k) = \phi^{-1}(\xi\phi(F))(k)
		  = L_\xi F(k) + \pi_V(\xi)F(k).  
	\end{equation*}	
%	*******Don't erase this part ever ********	
%	These are the right actions because $F\in\mu_*\pi^*\ms{V}(U)$ is a function 
%	\begin{equation*} F: K\to \prod_{k\in K} \ms{V}(k^{-1}U)\end{equation*} and so is $e_t^*(\psi)^{-1}F(kt)$.
%	Additionally, if $\phi(F): K\to \ms{D}_\lambda(U)$ is defined by $\phi(F)(k) = \nu(k).F(k)$,
%	so $\phi^*(\xi)^{-1}F = d\nu(\xi)F+L_\xiF: K\to \prod\ms{D}_\lambda(k^{-1}U)$,
%	and this determines $\phi^*(\xi)F$ for all $\ms{V}$.  Finally, since we act on $K$ by 
% 	$S_x$ by the right regular action, the inclusion $\fr{s}\to \fr{k}$ is twisted by a minus
%	sign, so $d\lambda_S - \lambda_\fr{k}|_\fr{s} = \omega_V$.  
	On the other hand, $ \ms{D}_\lambda^\pi =\ms{D}_K\boxtimes\ms{D}_\lambda$,
	so for all $T\in \ms{D}_\lambda$ and $F\in \mu_*\pi^*\ms{V}^\bul$, we have
	\begin{equation*}
		(\pi_\G(T)F)(k) = (\phi^{-1}\circ T\circ \phi)(F)(k) =\pi_V(\nu_D(k^{-1})T\nu_D(k)) F(k),
	\end{equation*}
	where $\nu_D$ denotes the $K$-action on $\ms{D}_\lambda$, and $\pi_V$ the 
	$\ms{D}_\lambda$-action on $\ms{V}^\bul$.  
	Since 
$		\sigma_k^*\mu_*\pi^*\ms{V}^\bul=\mu_*\pi^*\ms{V}^\bul$
	for all $k\in K$, the 
	$K$-action on $F\in \mu_*\pi^*\ms{V}^\bul$ is defined by
	$(\nu_\G(k)F)(k') = F(k^{-1}k')$.
	The $(\cl{D}_\lambda, K)$-actions on $\rm{Hom}_{(\fr{k},S_x,\cl{N}(\fr{s}_x))}^\bul(\cl{N}
	(\fr{k}),\rm{\bf R}\G(X,\mu_*\pi^*\ms{V}^\bul))$ are given as in the definition
	of the algebraic functor $\rm{Ind_h}$.  Likewise, the homotopy of actions $i$ is defined for a morphism 
	$f: \cl{N}(\fr{k})\to R(K)\otimes \rm{\bf R}\G(X,\ms{V}^\bul)$ of degree $\ell$ by 
	\begin{equation*}
		(i_\xi f)(n)(k) = (-1)^{\ell}f(i_{\rm{Ad}(k)\xi}n)(k)
	\end{equation*}
	for all $n\in\cl{N}(\fr{k})$, $k\in K$, where we take the $(\cl{U}(\fr{k}),K)$-module
	structure on $\cl{N}(\fr{k})$ from \S \ref{ForAK-cxs}.  Then, by Proposition \ref{equicommprop}.\ we have an isomorphism
	\begin{equation*}
		\rm{\bf R}\G(X,\G_{K,S_x}^\rm{geo}(\ms{V}^\bul)) \simeq 
		\G_{K,S_x}^\rm{equi}(\rm{\bf R}\G(X,\ms{V}^\bul))
	\end{equation*}
	of  $(\cl{D}_{[\lambda]}, K)$-complexes.	
\endproof

\begin{prop}
    If $\tau$ is a connection on $Q$ compatible with $\ms{D}^i$, then
	\begin{equation*}
		\G^\rm{geo}_{K,S_x}i_{x+}T_x\tau[d_Q] \simeq \tau.
	\end{equation*}
\end{prop}
\proof
	Denote the quotient and projection maps $q: K\to Q$ and $p: K\to x$ respectively
	and observe $q=\mu\circ i_x$.  Let $i_x: K \to K\times Q$
	be the lift of $i_x: x\to Q$ under the projection $\pi: K\times Q\to Q$.  
	Base change  for $\ms{D}$-modules then gives the equivalence
	\begin{equation*}
		\G^\rm{geo}_{K, S_x}i_{x+}[d_Q] = q_+^{S_x}\pi^\circ [-d_{S_x}].
	\end{equation*}
	Lemma \ref{equivariance}.\ provides an equivalence of categories
	\begin{equation*}	
		\xymatrix@C=4pc{
		\cl{C}^\rm{b}(\ms{D}^{i\circ q},  S_x) \ar@<.5ex>[r]^{q_+^{S_x}p^\circ[-d_{S_x}]} & 
			\cl{C}^\rm{b}(\ms{D}^i, K) \ar@<.5ex>[l]^{p_+^Kq^\circ[-d_K]}.
			}  
	\end{equation*}
	The corresponding equivalence for 
	the underlying $\ms{O}$-module structure is
	\begin{equation}\label{omodeq}
		\xymatrix@C=3pc{
			\cl{R}ep_{S_x} \ar@<.5ex>[r]^{q_*^{S_x}p^*} & 
			\cl{M}(\ms{O}_Q, K), \ar@<.5ex>[l]^{p_*^Kq^*}
			}  
	\end{equation}
	which exists since $\tau$ is a vector bundle.
	In this setting, we have $i_x^\circ \rm{Res}_{S_x}^K \tau = p_*^Kq^*\tau = 
	T_x\tau $.  
	The equivalence (\ref{omodeq}) thus implies there is a natural isomorphism
	$\tau \simeq q_+^{S_x}\pi^\circ T_x\tau$. 
	 
\endproof

\begin{prop}  With the above notation, we have
$ p^\circ \G^\rm{geo}_{K, S_x} = \G^\rm{geo}_{K, S_x}p^\circ$.
\end{prop}
\proof
	The inverse image $p^\circ$ obviously commutes with $\pi^\circ$, and
	it commutes with $\mu_+^{S_x}$ by base change since $p$ is $S_x$-equivariant. 
\endproof

%%%%%%%%%%%%%%%%%%%%%%%%%%%%%%%%%%%%%%%%%%
%%%%%%%%%%%%%%%%%%%%%%%%%%%%%%%%%%%%%%%%%%

%%%%%%%
\section{Cohomology of Derived Standard Modules}
 %%%%%%%
 
 %%%%%%%%
\subsection{The Embedding Theorem}\label{embthm}
%%%%%%%%

For a Harish-Chandra pair $(\fr{g}, K)$, 
let $X_\theta$ be a partial flag variety of type $\theta$.  Fix a tdo 
$\ms{D}_\lambda$ on $X_\theta$.  Let $X$ denote the full flag variety of $\fr{g}$.
There is a natural fibration \begin{equation*}p:X\to X_\theta\end{equation*} and a corresponding natural morphism  from 
$\cl{U}_{[\lambda]}^p:=\G(X,\ms{D}_\lambda^p)$
to $\cl{D}_{[\lambda]} := \G(X_\theta, \ms{D}_\lambda)$. We obtain a pull-back functor
\begin{equation*}
     p^*: \cl{M}(\cl{D}_{[\lambda]}) \to \cl{M}(\cl{U}_{[\lambda]}^p)
\end{equation*}
of modules over these rings in the usual way.
The pull-back $p^*$ is related to $p^\circ$ by the following theorem.  

\begin{thm}[Embedding Theorem \ref{embthmstatement}] 
	 The inverse image functor
	\begin{equation*}
		p^\circ: \cl{M}(\ms{D}_\lambda)\to \cl{M}(\ms{D}_{\lambda}^p)
	\end{equation*}
	is fully faithful for all $\lambda$, and for $\lambda$ anti-dominant, we have
	$\G\circ p^\circ = p^*\circ\G$.
\end{thm}

\begin{proof} 
	We will prove full faithfulness of $p^\circ$ by constructing a functor
	from $\cl{M}(\ms{D}_\lambda^p)\to \cl{M}(\ms{D}_\lambda)$ which is
	 quasi-inverse to $p^\circ$ when restricted to the essential image of $p^\circ$.
    Since $p$ is smooth,  the shifted direct image functor $p_+[-n]$ is right adjoint
    to $p^\circ$ on the derived categories,
    where $n$ is the dimension of the fibers of $p$.  That is, for any $\ms{V}\in 
    \cl{M}(\ms{D}_\lambda)$ there is a natural morphism of complexes
     \begin{equation*}
     	\rm{ad}: \ms{V}\to p_+p^\circ \ms{V}[-n].
    \end{equation*}
     As $p$ is a flat morphism between smooth projective varieties, the inverse image
     $p^\circ$ is exact on $\cl{M}(\ms{D}_\lambda)$.  
     Recall that for surjective submersions the direct image functor $p_+$ is given by
    \begin{equation*}
       \ms{V}\mapsto \rm{\bf R}p_*(\ms{V}\otimes_{(\ms{D}_\lambda^p)^\circ}\cl{T}_{X/X_\theta}^\bul(\ms{D}_\lambda^p)^\circ) = \rm{\bf R}p_*(\ms{V}\otimes_{\ms{O}_X}\omega_{X/X_\theta}
       \otimes_{(\ms{D}_\lambda^\circ)^p}\cl{T}_{X/X_\theta}^\bul(\ms{D}_\lambda^\circ)^p).
    \end{equation*}
    	For all the reductions, we will use the first presentation of $p_+$, although the 
	left $\ms{D}_\lambda$-module structure is obscured by this notation.    
	
	The relative tangent complex vanishes below the fiber dimension, which implies that
	for $\ms{V}\in\cl{M}(\ms{D}_\lambda)$ there is a standard truncation triangle
     \begin{equation*}
       \xymatrix@C=-4pc@R=2pc{ & \tau_{\geq 1}(p^\circ\ms{V}\otimes_{(\ms{D}_\lambda^p)^\circ}
       \cl{T}^\bul_{X/X_\theta}(\ms{D}_\lambda^p)^\circ)[-n] )
       \ar[dl]^{[1]} \\  \rm{Ker} \; d^{-n} \ar[rr] & &
       p^\circ\ms{V}\otimes_{(\ms{D}_\lambda^p)^\circ}
       \cl{T}^\bul_{X/X_\theta}(\ms{D}_\lambda^p)^\circ[-n]  \ar[ul]}       
    \end{equation*} 
 in the derived category, where $d^{\bul}$ is the differential in 
 $p^\circ\ms{V}\otimes_{(\ms{D}_\lambda^p)^\circ}\cl{T}^{\bul}_{X/X_\theta}$.
    Since $p_*$ is left exact, the long exact sequence obtained from applying 
    $\rm{\bf R}p_*$ to this triangle induces an isomorphism 
    \begin{equation*}
       p_* \rm{Ker}\; d^{-n} \simeq \rm{H}^{-n}(\rm{\bf R}p_*(p^\circ\ms{V}\otimes_{(\ms{D}_\lambda^p)^\circ}
       \cl{T}^\bul_{X/X_\theta}(\ms{D}_\lambda^p)^\circ))
    \end{equation*} 
     in degree $0$.  
   The adjointness morphism thus descends to cohomology
        \begin{equation*}
		\rm{H}^0(\rm{ad}): \ms{V}\to \rm{\bf R}^0p_*\rm{Ker}\;d^{-n}.
        \end{equation*}
        In fact, since $p_+[-n]$ is left exact, the morphism $\rm{H}^0(\rm{ad})$ is injective.  

	The remainder of the proof proves surjectivity of $\rm{H}^0(\rm{ad})$. 
	The projection formula for $\ms{O}$-modules produces the isomorphism
	\begin{equation*}
		\rm{\bf R}p_*(p^*\ms{V}\otimes\omega_{X/X_\theta}\otimes
		\cl{T}^\bul_{X/X_\theta}) \simeq \ms{V}\otimes \rm{\bf R}p_*
		(\omega_{X/X_\theta}\otimes\cl{T}^\bul_{X/X_\theta});
	\end{equation*}
	 it is also an isomorphism of left $\ms{D}_\lambda$-modules.  To see this, 
	 we examine the tensor product  $\omega_{X/X_\theta}\otimes 
	 \cl{T}^\bul_{X/X_\theta}$.  Define $F_x=p^{-1}(p(x))$ and let $\fr{b}_x$ be 
	 the Borel corresponding to $x$ and similarly $\fr{p}_x$ the parabolic 
	 corresponding to $p(x)$.  There is a short exact sequence
   	\begin{equation*}   
   		0\to \fr{p}_x/\fr{b}_x \to \fr{g}/\fr{b}_x\to \fr{g}/\fr{p}_x\to 0
   	\end{equation*}
   	of the tangent spaces.  From this sequence we obtain the isomorphisms
   	\begin{equation*}
   		T_x\ms{T}_{X/X_\theta} = T_x\ms{T}_{F_x} \simeq \bar{\fr{n}}_{\fr{l},x}\;\;
		\text{and}\;\; T_x\Omega_{X/X_\theta} = T_x\Omega_{F_x}\simeq \fr{n}_{\fr{l},x},
   	\end{equation*}
   	where $\fr{n}_{\fr{l},x}$ is the nilpotent subalgebra of a Levi factor $\fr{l}_x$ of 
	$\fr{p}_x$ consisting of positive root spaces of type $\theta$ and 
	$\bar{\fr{n}}_{\fr{l},x}$ is the opposite nilpotent subalgebra in $\fr{l}_x$.  In 
	particular, the relative canonical sheaf of $p$ is the homogeneous line bundle
   	\begin{equation*}
   		\omega_{X/X_\theta} = \ms{O}(2\rho_\theta).
	\end{equation*}  
  	Furthermore, since $\cl{T}_{X/X_\theta}^{-n} = \wedge^n\ms{T}_{X/X_\theta} =
	\ms{O}(-2\rho_\theta)$, the tensor product $\omega_{X/X_\theta}\otimes 
	\cl{T}_{X/X_\theta}^{-n} \simeq \ms{O}_X$ is $p_*$-acyclic.  
     	We claim additionally that for $k\neq n$ we have
     	\begin{equation}\label{dirimis0}
     		\rm{\bf R}p_*(\omega_{X/X_\theta}\otimes \cl{T}^k_{X/X_\theta})\simeq 0.
   	\end{equation}
     	Let $U\subset X_\theta$ be any open subset  trivializing $p$.  
     	Then $\ms{T}_{X/X_\theta}^{-k}|_{U\times F}$ is isomorphic to 
	$p_F^*\wedge^k\ms{T}_F$, where $p_F: U\times F\to F$ is the projection to 
	$F\simeq F_x$.  Similarly, $\omega_{X/X_\theta} = p_F^*\omega_F$. Then, we 
	have
	\begin{equation*}
     		p_F^*(\omega_F\otimes \wedge^k\ms{T}_F)(U\otimes F) =
	 	\ms{O}_X(U)\otimes\G(F,\omega_F\otimes \wedge^k\ms{T}_F).
     	\end{equation*}
     	To show (\ref{dirimis0}) holds, it is enough to show 
     	$\G(F,\omega_F\otimes \wedge^k\ms{T}_F)$ for all $k\neq n$.

  	Fix a Levi decomposition $B=HU$, let $\fr{h}$ be the Lie algebra of $H$, and
   	$\fr{n}=[\fr{b},\fr{b}]$ the   maximal nilpotent subalgebra of $\fr{b}$.  
       	Note that $\fr{n}$ is the Lie algebra of the unipotent subgroup $U$.
       	Let $\fr{g}_\alpha$ be the root space in $\fr{g}$ for root $\alpha$, and $\Delta^+$ 
	the subset of positive roots.  Let $\rho$, as usual, be the half-sum of positive roots.  

   	The sheaf $\omega_X\otimes\wedge^k\ms{T}_X$ is the sheaf of sections of the 
   	vector bundle   $G\times^B(\CC_{2\rho} \otimes \wedge^k\bar{\fr{n}})$, which 
	has a filtration 
   	\begin{equation*}
   		F_p	=G\times^B(\CC_{2\rho}\otimes\wedge^k\bar{U}^p\bar{\fr{n}}),
   	\end{equation*}
   	whose quotients   $F_p/F_{p+1}$  have as their sheaf of sections
   	\begin{equation*}
   		\ms{O}(2\rho)\otimes \bigoplus_{\alpha\in\wedge^k\Delta^+_p}\ms{O}(-\alpha),
   	\end{equation*}
   	where $\wedge^k\Delta^+_p$ is the set of weights appearing in 
   	\begin{equation*}
   		\wedge^k(\bar{U}^p\bar{\fr{n}}/\bar{U}^{p+1}\bar{\fr{n}}).
   	\end{equation*}
   	That is, it is the set of $k$-fold sums of distinct length $p$ positive roots.
   	Therefore, for $\alpha\in\wedge^k\Delta^+_p$, the difference $2\rho-\alpha$ is 
   	a sum of positive roots, and hence not anti-dominant.  By the Borel-Weil-Bott 
	theorem, we have $\G(X,\ms{O}(2\rho-\alpha))=0$.  
   	It follows that $\G(X,\omega_X\otimes \wedge^k\ms{T}_X)=0$ for all $k\neq n$.
   
	The previous discussion implies that $p_+p^\circ\ms{V}[-n]\simeq \ms{V}$. 
	In fact, we proved something stronger:
	\begin{equation*}
		p_+[-n]|_{p^\circ\cl{M}(\ms{D}_\lambda)} = p_*|_{p^\circ\cl{M}(\ms{D}_\lambda)}.
	\end{equation*}
	Additionally, the adjointness morphism $\rm{ad}: \ms{V}\to p_+p^\circ\ms{V}[-n]$ is 
	 the identity.  Therefore, 
	\begin{equation*}
		\rm{Hom}_{\ms{D}_\lambda^p}(p^\circ \ms{V}, p^\circ\ms{W}) 
		= \rm{Hom}_{\ms{D}_\lambda}(\ms{V}, p_+p^\circ\ms{W}[-n])
		= \rm{Hom}_{\ms{D}_\lambda}(\ms{V},\ms{W}).
	\end{equation*}

      	Finally, we address the issue of commutativity with $\rm{\bf R}\G$. The 
      	equivalence $\ms{V}\simeq p_*p^\circ\ms{V}$ implies we have isomorphisms
      	\begin{equation*}
         		\rm{\bf R}\G(X_\theta, \ms{V}) = \rm{\bf R}\G(X_\theta, p_*p^\circ\ms{V}) 
		\simeq \rm{\bf R}\G(X, p^\circ\ms{V}),
      	\end{equation*}
      	of complexes of vector spaces.  That the $\cl{D}_{[\lambda]}$-actions agree is 
	a consequence of local triviality of the fibration $p: X\to X_\theta$.  Locally, the 
	tdo $\ms{D}_\lambda^p = \ms{D}_F \boxtimes \ms{D}_\lambda$ acts on 
	$p^\circ\ms{V} = \ms{O}_F\boxtimes\ms{V}$ factor-wise.   Therefore, the actions
	of $\cl{D}_{[\lambda]}$ and $\cl{U}_{[\lambda]}^p$ on 
	$\rm{\bf R}\G(X_\theta, \ms{V})$ are related by $p^*$.  
  	\end{proof}
  
  	\begin{cor}  The inverse image functor
	\begin{equation*}
		p^\circ: \rm{D^b}(\ms{D}_\lambda, K) \to \rm{D^b}(\ms{D}_\lambda^p, K)
	\end{equation*}	
	is fully faithful for all $\lambda$, and $\rm{\bf R}\G\circ p^\circ = p^*\circ \rm{\bf R}\G$.  
	\end{cor}
  	\begin{proof}
	The projection $p$ is $K$-equivariant.  Moreover, the proof of the theorem lifts to
	$\cl{M}_{w}(\ms{D}_\lambda,K)$.  We showed in \S \ref{gzfunctsec} that 
	the adjoint pair $p^\circ\dashv \rm{\bf R}^{-n}p_+$ on $\cl{M}_{w}(\ms{D}_\lambda,K)$ 
	defines an adjoint pair 
	\begin{equation*}\xymatrix{
		\rm{D^b}(\ms{D}_\lambda, K) \ar@<.5ex>[r]^{p^\circ} & 
		\rm{D^b}(\ms{D}_\lambda^p, K) \ar@<.5ex>[l]^{p_+[-n]}.}
	\end{equation*} 
	Since $p^\circ$ is fully faithful on $\cl{M}_{w}(\ms{D}_\lambda,K)$, the reduction 
	principle implies it is also fully faithful on $\rm{D^b}(\ms{D}_\lambda, K)$.

	The proof of the derived commutation property $\rm{\bf R}\G\circ p^\circ = p^*\circ 
	\rm{\bf R}\G$ is the same as the proof for abelian categories.  
	\end{proof}

	\begin{cor}
	For $\lambda\in\fr{h}_\theta^*$ anti-dominant, the functor 
	$\G: \cl{M}(\ms{D}_\lambda)\to 
	\cl{M}(\cl{D}_\lambda)$ is exact.  If $\lambda$ is also regular, then $\G$ is 
	also faithful.  
	\end{cor}

%%%%%
\subsection{Main Theorem}\label{mainthmsec}
%%%%%

 Let $X_\theta$ be the partial flag variety of $\fr{g}$
of type $\theta$.  Label the inclusion maps
\begin{equation}\label{basicmaps}
    \xymatrix{
        x \ar[r]_{i_x} \ar@/^1pc/[rr]^{j_x} & Q \ar[r]_i & X_\theta.
    }
\end{equation}
Let $p: X\to X_\theta$ be the natural fibration of the full flag variety $X$ over 
$X_\theta$.  From (\ref{basicmaps}) we obtain the following fiber products:
\begin{equation*}
    \xymatrix{
        F_x \ar[d]^{p} \ar[r]_{i_x} \ar@/^1pc/[rr]^{j_x} & 
        F_Q \ar[d]^{p} \ar[r]_i & X\ar[d]^{p} \\
        x \ar[r]^{i_x} \ar@/_1pc/[rr]_{j_x} & Q \ar[r]^i & X_\theta.
    }
\end{equation*}

For a fixed point $x\in Q$, let $\fr{p}_x$ be the parabolic subalgebra of $\fr{g}$ 
corresponding to $x$, and let $S_x$ be the stabilizer of $x$ in $K$.  Let $\fr{n}_x$ 
be the nilpotent  subalgebra of $\fr{p}_x$.  For a $(\fr{p}_x,S_x)$-module  $Z$, define 
  $Z^\#:=Z\otimes\wedge^{d_{X_\theta}}\bar{\fr{n}}_x$ as
  a $(\fr{p}_x,S_x)$-module via the adjoint action on the twisting factor.
  For a $(\fr{p}_x,S_x)$-module $V$, the induced $(\fr{g},S_x)$-module is
$\rm{ind}_{\fr{p}_x,S_x}^{\fr{g},S_x}(V) := \cl{U}(\fr{g})\otimes_{\cl{U}(\fr{p}_x)}V$.

\begin{thm}[Main Theorem \ref{mainthm}] Let $\ms{D}_\lambda$ be a homogeneous
 tdo on $X_\theta$ and $\tau$ a connection on a $K$-orbit $Q$ compatible 
 with $\lambda+\rho_n$.    Then there is a quasi-isomorphism 
\begin{equation*}
   \rm{\bf R}\G(X,p^\circ i_+ \tau) \simeq \G^\rm{equi}_{K,S_x}
   (\rm{ind}_{\fr{p}_x,S_x}^{\fr{g},S_x}(T_x\tau^\#))[d_Q]
\end{equation*}
in $\rm{D^b}(\cl{U}_{[\lambda-\rho_\theta]}, K)$.  
Upon taking cohomology, there is a convergent spectral sequence 
\begin{equation}\label{}
     \rm{\bf R}^p\G(X,p^\circ \rm{\bf R}^qi_+\tau) \Longrightarrow 
     \rm{\bf R}^{p+q+d_Q}\G_{K,S_x}(\rm{ind}_{\fr{p}_x, S_x}^{\fr{g},S_x}
     (T_x\tau^\#)).     
\end{equation}  
\end{thm}
\begin{proof} 
The results of \S \ref{gzfunctsec} define  $\G_{K,S_x}^\rm{geo}$
in the general derived equivariant setting and establish the isomorphism
$i_+\tau \simeq \G^{\rm{geo}}_{K,S_x}j_{x+}T_x\tau[d_Q]$ and 
commutativity properties ($\G_{K,S_x}^{\rm{geo}}i_+=i_+\G_{K,S_x}$ and $\rm{\bf R}\G\circ \G_{K,S_x}^{\rm{geo}}= \G_{K,S_x}^{\rm{equi}}\circ \rm{\bf R}\G$),  which culminate in the natural isomorphisms
\begin{equation} \label{result}
   \begin{array}{rcl}
   \rm{\bf R}\G(X, p^\circ i_+(\text{ - })) & \simeq &
    \rm{\bf R}\G(X,p^\circ \G_{K,S_x}^\rm{geo} j_{x+}T_x(\text{ - }))[d_Q] \\
   &\simeq &\G_{K,S_x}^\rm{equi} \rm{\bf R}\G(X, p^\circ j_{x+}T_x(\text{ - }))[d_Q].
   \end{array}
 \end{equation}
 Then by definition, we have $j_{x+}T_x\tau = j_{x*}(\ms{D}_{X_\theta\leftarrow x}\otimes T_x\tau)$, so 
 \begin{equation}
 \begin{array}{rcl}
   \rm{\bf R}\G(X, p^\circ i_+\tau) & \simeq & 
    \G_{K,S_x}^\rm{equi} (\ms{D}_{X_\theta\leftarrow x}\otimes T_x\tau)[d_Q] \\
   & = & \G_{K,S_x}^\rm{equi} (\cl{U}(\fr{g})/\fr{p}_x\cl{U}(\fr{g}) \otimes 
   T_x\omega_{X_\theta}^{-1}\otimes T_x\tau)[d_Q].  \\
   \end{array}
\end{equation}
Note that $T_x\omega_{X_\theta}^{-1} = \wedge^\rm{top}\bar{\fr{n}}_x$ and that
the parabolic $\fr{p}_x$ acts on $T_x\tau^\#$ by the $F_x$-invariant linear form
$\lambda-\rho_n$.  Therefore, there is an isomorphism
\begin{equation*}
	\cl{U}(\fr{g})/\fr{p}_x\cl{U}(\fr{g})\otimes T_x\tau^\# \simeq \rm{ind}_{\fr{p}_x,S_x}^{\fr{g},S_x}
	(T_x\tau^\#).
\end{equation*}
Finally, since $i$ is a locally closed immersion and so the composition of the derived
$\ms{O}_X$-module direct image $\rm{\bf R}i_*$ with an exact functor.  The spectral sequence is then seen to follow precisely from the Leray spectral sequence
$\rm{\bf R}^p\G\rm{\bf R}^qi_* \implies \rm{\bf R}^{p+q}(\G\circ i_*)$.
\end{proof}

When $i_+$ is exact, the left hand side of the spectral sequence collapses, and we have the following.
\begin{cor}  If $i_+$ is exact then the isomorphism
%\begin{equation}
 $    \rm{\bf R}^p\G(X,p^\circ i_+\tau) \simeq \rm{\bf R}^{p+d_Q}\G_{K,S_x}
     (\rm{ind}_{\fr{p}_x, S_x}^{\fr{g},S_x}(T_x\tau^\#))$ holds for all $p$.
%\end{equation}  
\end{cor}
\noindent This is the case for any orbit $Q$ in the full flag variety.  
Another example is for the open orbit in any partial flag variety of $\fr{g}$ if the Cartan 
involution defining $K$ is quasi-split. 
Alternatively, if we are working with twisted differential operators, but take $\lambda$ 
to be anti-dominant, then $\G$ is exact, and we again find the left hand
 side collapses.
\begin{cor}  For $\lambda$ anti-dominant, we have for all $q$ an isomorphism 
 $$    \G(X,p^\circ \rm{\bf R}^qi_+\tau) \simeq \rm{\bf R}^{q+d_Q}\G_{K,S_x}
     (\rm{ind}_{\fr{p}_x, S_x}^{\fr{g},S_x}
     ( T_x\tau^\#)).$$
\end{cor}
\noindent Finally, we can combine the two corollaries to obtain a third:
\begin{cor}
  For $\lambda$ anti-dominant and $i_+$ exact, we have
  \begin{equation*}
     \G(X,p^\circ i_+\tau) \simeq \rm{\bf R}^{d_Q}\G_{K,S_x}(\rm{ind}_{\fr{p}_x ,S_x}^{\fr{g},S_x}
     (T_x\tau^\#))
  \end{equation*}
  and all other derived Zuckerman functors vanish. 
\end{cor}

Of possibly the greatest significance is the fact that the convergent spectral
sequence (\ref{specseq}) determines equalities in the Grothendieck group
\begin{equation}\label{kgp}
    [\rm{\bf R}^{n+d_Q}\G_{K,S_x}(\rm{ind}_{\fr{p}_x,S_x}^{\fr{g},S_x}(T_x\tau^\#))] 
    = \sum_{p+q=n} [\rm{\bf R}^p\G(X,p^\circ \rm{\bf R}^qi_+\tau)],
\end{equation}
which may give additional geometric insight into the computation of composition 
series of degenerate principal series.  Low rank examples of applications of this result are discussed in Chapter 8.\ of \cite{k}. 

%%%%
\subsection{Duality and Cohomologically Induced Modules}
%%%%%

	Fix a Levi subgroup $L_x$ of $F_x$.  Then, $L_x\cap K$ is a maximal 
	reductive subgroup of $S_x$.  The representation $T_x\tau$ of the previous 
	section is really a representation of $L_x\cap K$ extended trivially to $S_x$.  For 
	such representations $V$, we have equality of underlying $\fr{g}$-modules
	\begin{equation*}
		\rm{ind}_{\fr{p}_x,S_x}^{\fr{g},S_x}(V) = \rm{ind}_{\fr{p},L_x}^{\fr{g},L_x}(V).  
	\end{equation*}
	Abstractly, let $\fr{p}\subset \fr{g}$ be any parabolic, and let $L\subset K$ be 
	a reductive subgroup such that $\fr{l}\subset \fr{p}$.  Then, the left adjoint to 
	the forgetful functor from $(\fr{g}, L)$-modules to $(\fr{p},L)$-modules is 
	$\rm{ind}_{\fr{p},L}^{\fr{g},L}$ and the right adjoint is
	\begin{equation*}
		\rm{pro}_{\fr{p},L}^{\fr{g},L}(\text{ - }) = 
		\rm{Hom}_{\fr{p}}(\cl{U}(\fr{g}), \text{ - })_{[L]},
	\end{equation*}  
	where the $[L]$ indicates that we take $L$-finite vectors.  For any 
	$(\fr{g},K)$-module $V$, define the contragredient $V^{\vee} =V^*_{[K]}$.  Then, 
	we have the following lemma.
\begin{lem}[\cite{hmsw}, Lemma 3.1]
	For $V$ any $(\fr{p},L)$-module, 
	\begin{equation*}
		\rm{ind}_{\fr{p},L}^{\fr{g},L}(V)^\vee = \rm{pro}_{\fr{p},L}^{\fr{g},L}(V^\vee).
	\end{equation*}
\end{lem}

To identify the modules of Theorem \ref{mainthm}.\ as contragredient to 
cohomologically induced modules, we need to introduce Zuckerman duality:

\begin{thm}
	Let $G$, $P$, and $L$ be as above and let $V$ be a $(\fr{l},L\cap K)$-module.
	 Let $\fr{n}$ be the nilradical of $\fr{p}$ and let $\fr{o} = \fr{k}\cap \fr{n}\oplus 
	 \fr{k}\cap \bar{\fr{n}}$ and $s = \dim{\fr{k}\cap\fr{n}}$.  Then, for all $i\geq 0$, 
	 there is an isomorphism of $(\fr{g},K)$-modules
	\begin{equation*}
		\G_{K,L\cap K}^i(V^\vee) \simeq \G_{K,L\cap K}^{2s-i}
		(\wedge^\rm{top}\fr{o}\otimes V)^\vee.
	\end{equation*}
\end{thm}
See \cite[Cor.\ 6.1.9]{hp} for the proof of this theorem in the case of real groups.  
The one-dimensional $(\fr{l}, L\cap K)$-module $\wedge^\rm{top}\fr{o}$ is trivial for $\fr{l}$,
but may have a non-trivial action of the component group of $L\cap K$.  
Additionally, we have
\begin{thm}[\cite{mp}, Theorem 1.13]\label{redrest}
	Let $S\subset K$ be a subgroup and $T$ its Levi factor. 
	The Zuckerman functor $\rm{\bf R}\G_{K,S}$ is the restriction of $\rm{\bf R}\G_{K,T}$
	to $\rm{D}(\cl{M}(\fr{g},S))$. 
\end{thm}

For any $(\fr{p}_x, L_x\cap K)$-module $V$, let $V^\sim = V\otimes \wedge^\rm{top}\fr{n}_x$.
Then,  we have $(V^\#)^\vee = (V^\vee)^\sim$.  The $i$th cohomologically induced module of $V$ is 
defined to be 
\begin{equation*}
      \cl{R}^i(V)  =  \rm{\bf R}^i\G_{K,L_x\cap K}(\rm{pro}_{\bar{\fr{p}}_x, L_x\cap K}^{\fr{g},L_x\cap K}(V^\sim)).
\end{equation*}
Properties can be found in \cite{hp} or \cite{kv}.  As a consequence of Theorem \ref{mainthm}, we have the following corollary.  

\begin{cor}
	With the same hypotheses as Theorem \ref{mainthm}, let 
	$\fr{o} = \fr{k}\cap \fr{n}_x\oplus \fr{k}\cap \bar{\fr{n}}_x$.  Then
	\begin{equation*}
		\rm{\bf R}^{d_Q+i}_{K,S_x}(\rm{ind}_{\fr{p}_x,L_x\cap K}^{\fr{g},L_x\cap K}
		(T_x\tau^\#))^\vee\simeq 
		\cl{R}^{d_Q-i}((T_x\tau^\vee\otimes \wedge^{2d_Q}\fr{o}^\vee)^\sim).
	\end{equation*}
\end{cor}
\begin{proof}
	The duality results  yield the isomorphism
	\begin{equation*}
		\rm{\bf R}^{d_Q+i}_{K,S_x}(\rm{ind}_{\fr{p}_x,L_x\cap K}^{\fr{g},L_x\cap K}(T_x\tau^\#))^\vee
		\simeq \rm{\bf R}^{d_Q-i}\G_{K,S_x}(\rm{pro}_{\bar{\fr{p}}_x,L_x\cap K}^{\fr{g},L_x\cap K}
		((T_x\tau^\#\otimes\wedge^{2d_Q}\fr{o})^\vee)).
	\end{equation*} 
	The observation that $(V^\#)^\vee = (V^\vee)^\sim$ for any $(\fr{p}_x,L_x\cap K)$-module
	implies
	\begin{equation*}
		(T_x\tau^\#\otimes\wedge^{2d_Q}\fr{o})^\vee= 
		(T_x\tau^\vee \otimes\wedge^{2d_Q}\fr{o}^\vee)^\sim,
	\end{equation*}
	and Theorem \ref{redrest}.\ completes the proof.
\end{proof}

%There is a distinguished family of cohomologically induced modules
%\begin{equation*}
%	A^{\fr{q}}(\lambda) = \cl{R}^s(\CC_\lambda),
%\end{equation*}
%where $\CC_\lambda$ is the one-dimensional representation of $(\fr{l}_x,L_x\cap K)$ with
%the center of $\fr{l}_x$ acting by $\lambda$ and $s= \dim{\fr{n}_x\cap\fr{k}}.$  
%If $T_x\tau = \CC_\lambda\otimes \wedge^{2d_Q}\fr{o}^\vee$, then the main theorem
%produces modules of the form 
%\begin{equation*}
%	\cl{R}^{d_Q}(\CC_{-\lambda+2\rho_n}) =A^{\bar{\fr{p}}_x}(-\lambda+2\rho_n). 
%\end{equation*}

%In general, $T_x\tau = \xi\otimes \CC_\lambda$ for some $\lambda\in\fr{h}_\theta^*$ and
%irreducible representation $\xi$ of the component group of $L_x\cap K$.  Let 
%$V^\xi = \xi^\vee\otimes \wedge^{2d_Q}\fr{o}^\vee$.  Then, the main theorem produces
%modules of the form
%\begin{equation*}
%	\cl{R}^{d_Q-i}(V^\xi\otimes\CC_{-\lambda+2\rho_n}),
%\end{equation*}
% which by Proposition 11.47 of \cite{kv} is the underlying Harish-Chandra module for 
% the continuous series representation for $V^\xi\otimes\CC_{-\lambda+2\rho_n}$ when $i=d_Q$.  

%%%%%%%%%%%%%%%%%%%%%%%%%%%%%%%%%%%%%%%%%%

\bibliographystyle{amsplain}
\bibliography{biblist}

\end{document}